\documentclass{amsart}

\usepackage{amssymb}
\usepackage{amsthm}
\usepackage{amsmath}

\usepackage[normalem]{ulem}

\usepackage[usenames]{color}
\usepackage{todonotes}

\usepackage{hyperref}
\usepackage[capitalise]{cleveref}
\usepackage[shortlabels]{enumitem}

\hypersetup{
	colorlinks=true,
	linkcolor=blue,
	urlcolor=blue,
	citecolor=blue,
}

\theoremstyle{plain}
\newtheorem{proposition}{Proposition}[section]
\newtheorem{theorem}[proposition]{Theorem}
\newtheorem{lemma}[proposition]{Lemma}
\newtheorem{corollary}[proposition]{Corollary}
\theoremstyle{definition}
\newtheorem{example}[proposition]{Example}
\newtheorem{definition}[proposition]{Definition}
\newtheorem{observation}[proposition]{Observation}
\theoremstyle{remark}
\newtheorem{remark}[proposition]{Remark}

\newtheorem{question}[proposition]{Question}

\newtheorem*{example*}{Example}

\DeclareMathOperator{\Aut}{Aut}

\DeclareMathOperator{\diam}{diam}

\DeclareMathOperator{\PGL}{PGL}

\DeclareMathOperator{\Spanset}{Span}

\DeclareMathOperator{\Haus}{Haus} 
\DeclareMathOperator{\CAT}{CAT}

\DeclareMathOperator{\simplexcc}{max}
\DeclareMathOperator{\Stab}{Stab}

\DeclareMathOperator{\cdim}{c-dim}
\DeclareMathOperator{\relint}{relint}

\DeclareMathOperator{\dist}{d}
\DeclareMathOperator{\hil}{d_{\Omega}}
\DeclareMathOperator{\partiali}{\partial_i}

\DeclareMathOperator{\Bc}{\mathcal{B}}
\DeclareMathOperator{\Cc}{\mathcal{C}}

\DeclareMathOperator{\Kc}{\mathcal{K}}
\DeclareMathOperator{\Lc}{\mathcal{L}}
\DeclareMathOperator{\Nc}{\mathcal{N}}

\DeclareMathOperator{\Sc}{\mathcal{S}}

\DeclareMathOperator{\Uc}{\mathcal{U}}

\DeclareMathOperator{\Hb}{\mathbb{H}}

\DeclareMathOperator{\Nb}{\mathbb{N}}
\DeclareMathOperator{\Pb}{\mathbb{P}}
\DeclareMathOperator{\Rb}{\mathbb{R}}

\DeclareMathOperator{\Zb}{\mathbb{Z}}
\DeclareMathOperator{\Qb}{\mathbb{Q}}

\newcommand{\abs}[1]{\left|#1\right|}

\begin{document}

\title{Convex co-compact groups with one dimensional boundary faces}
\author{Mitul Islam}\curraddr{Max Planck Institute for Mathematics in the Sciences, 04103 Leipzig}\address{Department of Mathematics, University of Michigan, Ann Arbor, MI 48109.}
\email{mitulmi@umich.edu}
\author{Andrew Zimmer}\address{Department of Mathematics, University of Wisconsin-Madison, Madison, WI 53706.}
\email{amzimmer2@wisc.edu}
\date{\today}
\subjclass[2010]{}

\begin{abstract} 
In this paper we consider convex co-compact subgroups of the projective linear group. We prove that such a group is relatively hyperbolic with respect to a collection of virtually Abelian subgroups of rank two if and only if each open face in the ideal boundary has dimension at most one. We also introduce the ``coarse Hilbert dimension'' of a subset of a convex set and use it to characterize when a naive convex co-compact subgroup is word hyperbolic or relatively hyperbolic with respect to a collection of virtually Abelian subgroups of rank two.
\end{abstract}

\maketitle

\section{Introduction}

In this paper we consider the class of (naive) convex co-compact subgroups of $\PGL_d(\Rb)$, as defined in~\cite{DGF2017}. In earlier work~\cite{IZ2019b} we proved a general, geometric characterization of when such a group is relatively hyperbolic with respect to a (possibly empty) collection of virtually Abelian subgroups of rank at least two. In this paper, we specialize to the case of virtually Abelian subgroups of rank exactly two and provide a very simple (to state) characterization in terms of the ideal boundary of the associated convex hull. There are many examples of such convex co-compact groups coming from Coxeter groups and also from deformations of hyperbolic structures on certain cusped 3-manifolds followed by a doubling construction (see \cite{B2006}, \cite{BDL2018}, and \cite[Section 12.2]{DGF2017}).

To state our results precisely we need to introduce some terminology. Given a properly convex domain $\Omega \subset \Pb(\Rb^d)$, the \emph{automorphism group of $\Omega$} is defined to be
\begin{align*}
\Aut(\Omega) : = \{ g \in \PGL_d(\Rb) : g \Omega = \Omega\}.
\end{align*}
Then for a subgroup $\Gamma \subset \Aut(\Omega)$, the  \emph{full orbital limit set of $\Gamma$ in $\Omega$} is defined to be
\begin{equation*}
\Lc_{\Omega}(\Gamma):= \bigcup_{p \in \Omega} \Big( \overline{\Gamma \cdot p} \setminus \Gamma \cdot p \Big).
\end{equation*}
Next, let $\Cc_\Omega(\Gamma)$ denote the convex hull of $\Lc_\Omega(\Gamma)$ in $\Omega$. Then, convex co-compact subgroups can be defined as follows. 

\begin{definition}\cite[Definition 1.10]{DGF2017} \ \begin{enumerate}
\item Suppose $\Omega \subset \Pb(\Rb^d)$ is a properly convex domain, then an infinite discrete subgroup $\Gamma \subset \Aut(\Omega)$ is called \emph{convex co-compact} when $\Cc_\Omega(\Gamma)$ is non-empty and $\Gamma$ acts co-compactly on $\Cc_\Omega(\Gamma)$. 
\item A subgroup $\Gamma \subset \PGL_d(\Rb)$ is \emph{convex co-compact} if there exists a properly convex domain  $\Omega \subset \Pb(\Rb^d)$ where $\Gamma \subset \Aut(\Omega)$ is a convex co-compact subgroup. 
\end{enumerate}
\end{definition}

When $\Gamma$ is word hyperbolic there is a close connection between this class of discrete groups in $\PGL_d(\Rb)$ and Anosov representations, see~\cite{DGF2017} for details and~\cite{DGF2018,Z2017} for related results. Further, by adapting an argument of Benoist~\cite{B2004}, Danciger--Gu{\'e}ritaud--Kassel established a characterization of hyperbolicity in terms of the geometry of $\Cc_\Omega(\Gamma)$. To state their result we need some more definitions.

\begin{definition} A subset $S \subset \Pb(\Rb^d)$ is a \emph{simplex} if there exist $g \in \PGL_d(\Rb)$ and $0 \leq k \leq d-1$ such that 
\begin{align*}
gS = \left\{ [x_1:\dots:x_{k+1}:0:\dots:0] \in \Pb(\Rb^d): x_1>0,\dots, x_{k+1}>0  \right\}.
\end{align*}  
Then the \emph{dimension of $S$}, denoted $\dim(S)$, is $k$ (notice that $S$ is homeomorphic to $\Rb^{k}$) and the $(k+1)$ points 
\begin{align*}
g^{-1}\{ [1:0:\dots:0], [0:1:0:\dots:0], \dots, [0:\dots:0:1:0:\dots:0] \} \subset \partial S
\end{align*}
are the \emph{vertices} of $S$. 
\end{definition}

\begin{definition}\label{defn:properly-embedded}
 Suppose $A \subset B \subset \Pb(\Rb^d)$. Then $A$ is \emph{properly embedded in $B$} if the inclusion map $A \hookrightarrow B$ is a proper map (relative to the subspace topology). 
\end{definition}

Finally, given a properly convex domain $\Omega \subset \Pb(\Rb^d)$ let $\dist_\Omega$ denote the Hilbert metric on $\Omega$ (see Section~\ref{subsec:Hilbert_metric} for the definition). 

\begin{theorem}[{Danciger--Gu{\'e}ritaud--Kassel~\cite[Theorem 1.15]{DGF2017}}]\label{thm:char_of_hyp_cc}  Suppose $\Omega \subset \Pb(\Rb^d)$ is a properly convex domain, $\Gamma \subset \Aut(\Omega)$ is convex co-compact, and $\Cc : = \Cc_\Omega(\Gamma)$. Then the following are equivalent: 
\begin{enumerate}
\item every point in $\overline{\Cc} \cap \partial \Omega$ is an extreme point of $\Omega$, 
\item $\Cc$ does not contain a properly embedded simplex with dimension at least two, 
\item $(\Cc, \dist_\Omega)$ is Gromov hyperbolic, 
\item $\Gamma$ is word hyperbolic.
\end{enumerate}
\end{theorem}

\begin{remark}  In the special case when  $\Gamma$ acts co-compactly on $\Omega$, Theorem~\ref{thm:char_of_hyp_cc} is due to Benoist~\cite{B2004} and the proof in~\cite{DGF2017} follows similar arguments. 
\end{remark}

In this paper we establish a similar theorem for groups which are relatively hyperbolic with respect to a collection of virtually Abelian subgroups of rank two. To state our main result precisely, we introduce the following notation: given a properly convex domain $\Omega \subset \Pb(\Rb^d)$ and $x \in \overline{\Omega}$ let $F_\Omega(x)$ denote the \emph{(open) face} of $x$, that is 
\begin{align*}
F_\Omega(x) = \{ x\} \cup \left\{ y \in \overline{\Omega} : \text{ $\exists$ an open line segment in $\overline{\Omega}$ containing $x$ and $y$} \right\}.
\end{align*}
When $x \in \partial \Omega$, we say that $F_\Omega(x)$ is a \emph{boundary face} of $\partial \Omega$. Notice that $F_\Omega(x) = \Omega$ when $x \in \Omega$ and $F_\Omega(x) = \{x\}$ when $x \in \partial \Omega$ is an extreme point.  

\begin{theorem}\label{thm:cc_one_d_faces}(see Section~\ref{sec:pf_of_ond_d_faces_cc}) Suppose $\Omega \subset \Pb(\Rb^d)$ is a properly convex domain, $\Gamma \subset \Aut(\Omega)$ is convex co-compact, and $\Cc : = \Cc_\Omega(\Gamma)$. Then the following are equivalent: 
\begin{enumerate}
\item every boundary face of $\Omega$ which intersects $\overline{\Cc}$ has dimension at most one,
\item the collection of all properly embedded simplices in $\Cc$ with dimension two is closed and discrete in the local Hausdorff convergence topology induced by $\dist_\Omega$, 
\item $(\Cc,\dist_\Omega)$ is relatively hyperbolic with respect to a (possibly empty) collection of two dimensional properly embedded simplices, 
\item $\Gamma$ is a relatively hyperbolic group with respect to a (possibly empty) collection of virtually Abelian subgroups of rank two.
\end{enumerate}
\end{theorem}

\begin{remark} The implications $(2) \Leftrightarrow (3)\Leftrightarrow (4) \Rightarrow (1)$ follow easily from the general results in~\cite{IZ2019b} and so the difficulty is showing that $(1) \Rightarrow (2/3/4)$. \end{remark}

\begin{remark} There are a number of other results in the literature concerning relatively hyperbolic groups acting on properly convex domains, see for instance~\cite{CM2014, CLT2015, Choi2017, Choi2017b, Choi_book,Weisman2020} (we note the authors of~\cite{CM2014} are currently preparing an erratum for their paper). With the exception of~\cite{Weisman2020}, these results consider the case when $\Gamma \backslash \Cc$ is non-compact and  $\Gamma$ is relatively hyperbolic with respect to the fundamental groups of the ends (under some geometric assumptions on the ends and $\Cc$). There is some similarity between Theorem~\ref{thm:cc_one_d_faces}  and the statements in~\cite{Choi2017, Choi2017b, Choi_book}, but to the best of our knowledge, there is no non-trivial mathematical overlap between the results. 
\end{remark}

Theorem~\ref{thm:cc_one_d_faces} can be viewed as an extension of the following result of Benoist. 

\begin{theorem}[Benoist~\cite{B2006}] If $M$ is a closed irreducible orientable 3-manifold and $M$ admits a convex real projective structure, then either
\begin{enumerate}
\item $M$ is geometric with geometry $\Rb^3$, $\Rb \times \Hb^2$, or $\Hb^3$,
\item $M$ is non-geometric and every component in the geometric decomposition is hyperbolic.
\end{enumerate}
\end{theorem}

Using Benoist's theorem, one can deduce the following special case of Theorem~\ref{thm:cc_one_d_faces}.

\begin{corollary}[to Benoist's result] Suppose $\Omega \subset \Pb(\Rb^4)$ is a properly convex domain and $\Gamma \subset \Aut(\Omega)$ is a discrete group which acts co-compactly on $\Omega$. If every boundary face of $\Omega$ has dimension at most one, then $\Gamma$ is relatively hyperbolic with respect to a (possibly empty) collection of virtually Abelian subgroups of rank two.
\end{corollary}

In fact, using the theory of 3-manifolds and relatively hyperbolic groups, one can deduce Benoist's theorem  from the above corollary and so Theorem~\ref{thm:cc_one_d_faces} can be viewed as an extension of this restated version of Benoist's theorem.

Theorem~\ref{thm:cc_one_d_faces} also provides a partial answer to a question asked by Choi--Lee--Marquis.

\begin{question}\cite[Remark 1.11]{CLM2016}  Suppose $\Omega \subset \Pb(\Rb^d)$ is a properly convex domain and $\Gamma \subset \Aut(\Omega)$ is a discrete group which acts co-compactly on $\Omega$. If $\Omega$ is irreducible and non-symmetric, is $\Gamma$ relatively hyperbolic with respect to a (possibly empty) collection of virtually Abelian subgroups of rank at least two? \end{question}

Theorem~\ref{thm:cc_one_d_faces} says the answer is yes when every boundary face of $\Omega$ has dimension at most one.

\subsection{Naive convex co-compact subgroups} We will also prove a version of Theorem~\ref{thm:cc_one_d_faces} for naive convex co-compact subgroups. This is a larger class of groups and as such the result is, by necessity, more technical. 

\begin{definition}\label{defn:cc_naive}  Suppose $\Omega \subset \Pb(\Rb^d)$ is a properly convex domain. An infinite discrete subgroup $\Gamma \subset \Aut(\Omega)$ is called \emph{naive convex co-compact} if there exists a non-empty closed convex subset $\Cc \subset \Omega$ such that 
\begin{enumerate}
\item $\Cc$ is $\Gamma$-invariant, that is, $g\Cc = \Cc$ for all $g \in \Gamma$, and
\item $\Gamma$ acts co-compactly on $\Cc$. 
\end{enumerate}
In this case, we say that $(\Omega, \Cc, \Gamma)$ is a \emph{naive convex co-compact triple}. 
\end{definition}

It is straightforward to construct examples where $\Gamma \subset \Aut(\Omega)$ is naive convex co-compact, but not convex co-compact (see for instance \cite[Section 2.3]{IZ2019b}). In these cases, the convex subset $\Cc$ in Definition~\ref{defn:cc_naive} is a strict subset of $\Cc_\Omega(\Gamma)$.

One key difference between convex co-compact and naive convex co-compact subgroups is the following:  If $\Gamma \subset \Aut(\Omega)$ is a convex co-compact subgroup and $\overline{\Cc_\Omega(\Gamma)}$ intersects an open boundary face $F$ of $\partial \Omega$, then $F \subset \overline{\Cc_\Omega(\Gamma)}$; see for instance \cite[Section 4]{DGF2017}. However, if $(\Omega, \Cc, \Gamma)$ is a naive convex co-compact triple, then it is possible for $\overline{\Cc}$ to intersect a boundary face without containing it entirely; see the following example.

\begin{example*}
Consider $\Omega:=\{ [x_1 : x_2 : x_3] : x_1, x_2 , x_3 >0\}$, $\Cc:=\{[x_1:y:y] : x_1,y >0\}$, and $\Gamma:=\left\langle \begin{bmatrix} 2 & 0 & 0 \\0 & 1 & 0 \\ 0 & 0 & 1 \end{bmatrix}\right\rangle$. Then $(\Omega, \Cc, \Gamma)$ is a naive convex co-compact triple. Further  $F_{\Omega}([0:1:1]) \cap \overline{\Cc}=\{ [0:1:1]\}$ while 
$$
F_{\Omega}([0:1:1])=\{ [0: x_2:x_3] : x_2,x_3>0\} \not \subset \overline{\Cc}.
$$
\end{example*}

So when studying naive convex co-compact subgroups, it is not enough to consider the dimension of the boundary faces of $\Omega$ which intersect the closure of convex subset $\Cc$, but the ``size'' of $\overline{\Cc}$ in each boundary face.

To make ``size'' precise we introduce the following definition. 

\begin{definition} 
\label{defn:cdim}
Suppose $\Omega \subset \Pb(\Rb^d)$ is properly convex and open in its span. Then the \emph{coarse dimension} of a non-empty subset $A \subset \Omega$, denoted by $\cdim_\Omega(A)$, is the smallest integer $k \geq 0$ such that there exist $R > 0$ and a $k$-dimensional convex subset $B \subset \Omega$ such that
\begin{align*}
A \subset \Nc_{\Omega}(B;R):=\{ p \in \Omega : \dist_\Omega(p,B) < R\}
\end{align*}
where $\dist_\Omega$ is the Hilbert metric on $\Omega$. In the extremal case when $\Omega$ is a point, we define $\cdim_\Omega(\Omega):=0$. 
\end{definition}

\begin{example*} 
Suppose $\Omega$ and $\Cc$ are as in the previous example. Then, for any $r>0$,  $\cdim_{\Omega}\left( \Nc_{\Omega}(\Cc;r) \right)=1$ and $$\cdim_{F_{\Omega}([0:1:1])} \left( \overline{\Nc_{\Omega}(\Cc;r)} \cap F_{\Omega}([0:1:1])   \right)=0.$$
\end{example*} 

We will show that the coarse dimension of boundary faces can be used to characterize word hyperbolic naive convex co-compact subgroups. 

\begin{theorem}\label{thm:ncc_GH}(see Section~\ref{sec:pf_of_GH_ncc}) Suppose $(\Omega, \Cc, \Gamma)$ is a naive convex co-compact triple. Then the following are equivalent: 
\begin{enumerate}
\item $\cdim_{F_\Omega(x)} \left( \overline{\Cc} \cap F_\Omega(x) \right)=0$ for all $x \in \overline{\Cc} \cap \partial \Omega$,
\item $\Cc$ does not contain a properly embedded simplex with dimension at least two, 
\item $(\Cc, \dist_\Omega)$ is Gromov hyperbolic, 
\item $\Gamma$ is a word hyperbolic group.
\end{enumerate}
\end{theorem}

\begin{remark} Recall, if $x \in \partial \Omega$ is an extreme point, then $F_\Omega(x) = \{x\}$ and so $\dim F_\Omega(x) = 0$. Hence, Theorem~\ref{thm:ncc_GH} is a naive convex co-compact analog of Theorem~\ref{thm:char_of_hyp_cc}.
\end{remark}

For naive convex co-compact subgroups we also prove the following analog of Theorem~\ref{thm:cc_one_d_faces}. 

\begin{theorem}\label{thm:ncc_one_d_faces}(see Section~\ref{sec:pf_of_ond_d_faces_ncc}) Suppose $(\Omega, \Cc, \Gamma)$ is a naive convex co-compact triple. Then the following are equivalent: 
\begin{enumerate}
\item $\cdim_{F_\Omega(x)} \left(\overline{\Cc} \cap F_\Omega(x)\right) \leq 1$ for all $x \in \overline{\Cc}\cap \partial \Omega$,
\item $(\Cc,\dist_\Omega)$ is relatively hyperbolic with respect to a (possibly empty) collection of two dimensional properly embedded simplices, 
\item $\Gamma$ is a relatively hyperbolic group with respect to a (possibly empty) collection of virtually Abelian subgroups of rank two.
\end{enumerate}
\end{theorem}

\subsection*{Acknowledgements} M. Islam thanks Louisiana State University for hospitality during a visit where work on this project started. 

M. Islam was partially supported by grant DMS-1607260 from the
National Science Foundation. A. Zimmer was partially supported by a Sloan research fellowship and grants DMS-2105580 and DMS-2104381 from the National Science Foundation.

The authors thank the anonymous referee(s) for their comments and suggestions.

\section{Preliminaries}\label{sec:prelim}
 
 \subsection{Convexity}
 
 In this section we recall some standard definitions related to convexity in real projective space. 

\begin{definition} \ 
\begin{enumerate}
\item A subset $C \subset \Pb(\Rb^d)$ is \emph{convex} if there exists an affine chart $\mathbb{A}$ of $\Pb(\Rb^d)$ where $C \subset \mathbb{A}$ is a convex subset. 
\item A subset $C \subset \Pb(\Rb^d)$ is \emph{properly convex} if there exists an affine chart $\mathbb{A}$ of $\Pb(\Rb^d)$ where $C \subset \mathbb{A}$ is a bounded convex subset. 
\item When $C$ is a properly convex set which is open in $\Pb(\Rb^d)$ we say that $C$ is a \emph{properly convex domain}.
\end{enumerate}
\end{definition}

Notice that if $C \subset \Pb(\Rb^d)$ is convex, then $C$ is a convex subset of every affine chart that contains it. 

A \emph{line segment} in $\Pb(\Rb^{d})$ is a connected subset of a projective line. Given two points $x,y \in \Pb(\Rb^{d})$ there is no canonical line segment with endpoints $x$ and $y$, but we will use the following convention: if $C \subset \Pb(\Rb^d)$ is a properly convex set and $x,y \in \overline{C}$, then (when the context is clear) we will let $[x,y]$ denote the closed line segment joining $x$ to $y$ which is contained in $\overline{C}$. In this case, we will also let $(x,y)=[x,y]\setminus\{x,y\}$, $[x,y)=[x,y]\setminus\{y\}$, and $(x,y]=[x,y]\setminus\{x\}$.

Along similar lines, given a properly convex subset $C \subset \Pb(\Rb^d)$ and a subset $X \subset C$ we will let 
\begin{align*}
{\rm ConvHull}_C(X)
\end{align*}
 denote the smallest convex subset of $C$ which contains $X$.

If $V \subset \Rb^d$ is a non-zero linear subspace, we will let $\Pb(V) \subset \Pb(\Rb^d)$ denote its projectivization. For a non-empty set $X \subset \Pb(\Rb^d)$, $\Pb(\Spanset(X))$ is the projectivization of the linear span of $X$. 

We also make the following topological definitions.

\begin{definition}\label{defn:topology} Suppose $C \subset \Pb(\Rb^d)$ is a properly convex set. The \emph{relative interior of $C$}, denoted by $\relint(C)$, is  the interior of $C$ in $\Pb( \Spanset C)$. In the case that $C = \relint(C)$, then $C$ is said to be \emph{open in its span}. The \emph{boundary of $C$} is $\partial C : = \overline{C} \setminus \relint(C)$, and the \emph{ideal boundary of $C$} is
\begin{align*}
\partiali C := \partial C \setminus C.
\end{align*}
Finally, we define $\dim C$ to be the dimension of $\relint(C)$ (notice that $\relint(C)$ is homeomorphic to $\Rb^{\dim C}$). 
\end{definition}

Recall from Definition \ref{defn:properly-embedded} that a subset $A \subset B \subset \Pb(\Rb^d)$ is properly embedded if the inclusion map $A \hookrightarrow B$ is proper. If $B$ is a properly convex set, then we have another characterization of properly embedded subsets using the notation in Definition~\ref{defn:topology}~: $A \subset B$ is properly embedded if and only if $\partiali A \subset \partiali B$.

\subsection{The Hilbert metric and faces}\label{subsec:Hilbert_metric} 

Suppose $\Omega \subset \Pb(\Rb^{d})$ is a properly convex domain. For distinct points $x,y \in \Omega$, let $\overline{xy}$ be the projective line containing them and let $a,b$ be the two points in $\overline{xy} \cap \partial\Omega$ ordered $a, x, y, b$ along $\overline{xy}$. Then the \emph{Hilbert distance} between $x$ and $y$ is defined to be
\begin{align*}
\dist_{\Omega}(x,y) = \frac{1}{2}\log [a, x,y, b]
\end{align*}
 where 
 \begin{align*}
 [a,x,y,b] = \frac{\abs{x-b}\abs{y-a}}{\abs{x-a}\abs{y-b}}
 \end{align*}
 is the projective cross ratio. It is a complete $\Aut(\Omega)$-invariant proper metric on $\Omega$ generating the standard topology on $\Omega$. Moreover, if $x, y \in \Omega$, the projective line segment  $[x,y]$  is a geodesic joining $x$ and $y$.
 
 For $x \in \Omega$ we will let 
 \begin{align*}
 \Bc_\Omega(x;r) := \{ y \in \Omega : \dist_\Omega(y,x) < r\}
 \end{align*}
 and for $A \subset \Omega$ we will let 
 \begin{align*}
 \Nc_\Omega(A;r) := \{ y \in \Omega : \dist_\Omega(y,A) < r\}.
 \end{align*}

Recall (from the introduction) that given a properly convex domain $\Omega \subset \Pb(\Rb^d)$ and $x \in \overline{\Omega}$ the open face of $x$ is 
\begin{align*}
F_\Omega(x) = \{ x\} \cup \left\{ y \in \overline{\Omega} : \text{ $\exists$ an open line segment in $\overline{\Omega}$ containing $x$ and $y$} \right\}.
\end{align*}
Given a subset $X \subset \overline{\Omega}$, we then define
\begin{align*}
F_{\Omega}(X):=\cup_{x \in X} F_{\Omega}(x).
\end{align*}

The following observations follow immediately from convexity and the definitions (also see \cref{sec:proof_of_obs_about_faces}).

\begin{observation}
\label{obs:faces} 
Suppose $\Omega \subset \Pb(\Rb^d)$ is a properly convex domain. 
\begin{enumerate}
\item $F_\Omega(x)$ is convex and open in its span,
\item $y \in F_\Omega(x)$ if and only if $x \in F_\Omega(y)$ if and only if $F_\Omega(x) = F_\Omega(y)$,
\item if $y \in \partial F_\Omega(x)$, then $F_\Omega(y) \subset \partial F_\Omega(x)$,
\item if $x, y \in \overline{\Omega}$, $z \in (x,y)$, $p \in F_{\Omega}(x)$, and $q \in F_{\Omega}(y)$, then 
\begin{align*}
(p,q) \subset F_\Omega(z).
\end{align*}
In particular, $(p,q) \subset \Omega$ if and only if $(x,y) \subset \Omega$.
\end{enumerate}
\end{observation}

Directly from the definition of the Hilbert metric one obtains the following. 

\begin{proposition}
\label{prop:dist_est_and_faces} Suppose $\Omega \subset \Pb(\Rb^d)$ is a properly convex domain, $(x_n)_{n \geq 1}$ is a sequence in $\Omega$, and $\lim_{n \to \infty} x_n =x \in \overline{\Omega}$. If $(y_n)_{n \geq 1}$ is another sequence in $\Omega$, $\lim_{n \to \infty} y_n =  y \in \overline{\Omega}$, and 
\begin{align*}
\liminf_{n \rightarrow \infty} \dist_\Omega(x_n,y_n) < + \infty,
\end{align*}
then $y \in F_\Omega(x)$ and 
\begin{align*}
\dist_{F_\Omega(x)}(x,y) \leq \liminf_{n \rightarrow \infty} \dist_\Omega(x_n,y_n).
\end{align*}
\end{proposition}

\subsection{The center of mass of a compact subset}

It is possible to define a ``center of mass'' for a compact set in a properly convex domain. Let $\Kc_d$ denote the set of all pairs $(\Omega, K)$ where $\Omega \subset \Pb(\Rb^d)$ is a properly convex domain and $K \subset \Omega$ is a compact subset. 

\begin{proposition}\label{prop:center_of_mass} There exists a function
\begin{align*}
(\Omega, K) \in \Kc_d \mapsto {\rm CoM}_\Omega(K) \in \Pb(\Rb^d)
\end{align*}
such that:
\begin{enumerate}
\item ${\rm CoM}_\Omega(K)  \in {\rm ConvHull}_\Omega(K)$, 
\item ${\rm CoM}_\Omega(K) = {\rm CoM}_\Omega({\rm ConvHull}_\Omega(K))$, and
\item if $g \in \PGL_d(\Rb)$, then $g{\rm CoM}_\Omega(K)={\rm CoM}_{g\Omega}(gK)$,
\end{enumerate}
for every $(\Omega, K) \in \Kc_d$. 
\end{proposition} 

\begin{proof} There are several constructions of such a center of mass, see for instance~\cite[Lemma 4.2]{L2014} or~\cite[Proposition 4.5]{IZ2019}. The approach in~\cite{IZ2019} is based on an argument of Frankel~\cite[Section 12]{Fra1989} in several complex variables. 
\end{proof} 

\subsection{The Hausdorff distance}

Recall that when $(X,\dist)$ is a metric space, the \emph{Hausdorff pseudo-distance} between two subsets $A,B \subset X$ is defined by 
\begin{align*}
\dist^{\Haus}(A,B) = \max \left\{ \sup_{a \in A} \inf_{b \in B} \dist(a,b), \ \sup_{b \in B} \inf_{a \in A} \dist(a,b) \right\}
\end{align*}
when $A$ and $B$ are both non-empty, and $\dist^{\Haus}(A,B) = \infty$ otherwise. 

The Hausdorff pseudo-distance is very useful when considering compact subsets: When $(X,\dist)$ is a complete metric space space,  $d^{\Haus}$ is a complete metric on the set of non-empty compact subsets of $X$. This pseudo-distance is less useful when dealing with closed sets, as the next example demonstrates. 

\begin{example} 
\label{ex:hauss_top_on_closed_sets_bad}
Consider $\Rb^2$ with the Euclidean distance. Let $B_n : = \overline{\Bc_{\Rb^2}((0,n); n)}$ be the closed ball of radius $n$ centered at $(0,n)$ and let $H: = \{ p=(x,y) \in \Rb^2 :  y \geq 0\}$ be the closed upper half plane. In any reasonable topology on closed sets one would like the sequence $B_n$ to converge to $H$. Unfortunately, with respect to the Hausdorff pseudo-distance one has, for all $n$,
$$
\dist^{\Haus}(B_n, H) = \infty.
$$
\end{example} 

\subsection{Local Hausdorff convergence topology}
In this section we recall a useful topology on the set of non-empty closed subsets of a metric space. This can be interpreted as a localization of the Hausdorff pseudo-distance that we discussed above. The topology we describe is a natural extension of the topology on compact subsets determined by the Hausdorff distance and has been used extensively in different areas of mathematics (e.g. see Hruska--Kleiner's~\cite{HK2005} work in $\CAT(0)$ geometry or Frankel's work in several complex variables~\cite{Fra1989}).

Let $\Cc(X)$ denote the set of all non-empty closed subset of a metric space $(X,\dist)$. For any $x \in X$ and $r>0$, we will denote the metric $r$-neighborhood of $x$ by 
\begin{align*}
\Bc_X(x;r) =\{ y\in X : \dist(x,y) < r\}.
\end{align*}
\begin{definition}
For a closed set $C_0 \subset X$, a base point $x_0 \in X$, and $r_0, \epsilon_0 > 0$ define the set $U(C_0,x_0,r_0,\epsilon_0)$ to consist of all closed subsets $C \subset X$ where
\begin{align*}
\dist^{\Haus}\Big(C_0 \cap \Bc_X(x_0;r_0),\, C \cap \Bc_X(x_0;r_0)\Big) < \epsilon_0.
\end{align*}
The \emph{local Hausdorff convergence topology on $\Cc(X)$} (induced by the metric $\dist$ on $X$) is the topology generated by the sets $U(\cdot, \cdot, \cdot, \cdot)$. 
\end{definition}
When the metric space $(X,\dist)$ is clear from context, we will often simply refer to this as the \emph{local Hausdorff topology induced by $\dist$} for brevity.

\begin{remark}
There are other well-known topologies on the space of non-empty closed subsets of a metric space, for instance the Chabauty topology \cite{IB2018,BC2013}. 
\end{remark}

\begin{example}
Assume the same set-up and notation as in \cref{ex:hauss_top_on_closed_sets_bad}. Then $B_n$ converges to $H$ in the local Hausdorff convergence topology on $\Cc(\Rb^2)$, see \cref{cor:convergence in the local Haus top ii} below. 
\end{example}

We note that when the metric space $(X,\dist)$ is proper, the local Hausdorff convergence topology is second countable. 

\begin{observation} 
If $(X,\dist)$ is a proper metric space, then the local Hausdorff convergence topology on $\Cc(X)$ is second countable. 
\end{observation} 

\begin{proof} Since $(X,\dist)$ is proper, it has a countable dense subset $A \subset X$. Fix an enumeration $\Qb \cap (0,\infty) = \{ r_n\}$. Then for each $n \in\Nb$ and $a \in A$, the set
$$
\Cc_{n, a} : =\{ K  : K \text{ compact and } K \subset \Bc_X(a;r_n)\}
$$
endowed with the Hausdorff distance is a compact metric space. Hence $\Cc_{n, a}$ has a countable dense subset $B_{n,a}$. Then 
$$
\{ U(C, a, r_n, m^{-1}) : a \in A, \ n, m \in \Nb, \ C \in B_{n,a} \}
$$
is a countable basis for the local Hausdorff convergence topology.
\end{proof} 

Based on the definition of the topology, one might expect that $C_n \rightarrow C$ if and only if 
 $$
 \lim_{n \rightarrow \infty} \dist^{\Haus}\Big( C_n \cap \Bc_{\Rb}(x_0;r), C \cap \Bc_{\Rb}(x_0; r) \Big) =0 
 $$ 
for all $x_0 \in X$ and $r > 0$. However, the next example demonstrates that one has to be careful with the choice of $x_0 \in X$ and $r > 0$. 

\begin{example*}
Consider $\Rb$ with the Euclidean distance. Let $C_n : = \{ 1/n\} \subset \Rb$ and $C := \{ 0\}$.  One can show that $C_n \to C$ in the local Hausdorff convergence topology (see Corollary~\ref{cor:convergence in the local Haus top ii}) however, if $x_0 = 1$ and $r = 1$, then 
$$
 \dist^{\Haus}\Big( C_n \cap \Bc_{\Rb}(x_0;r), C \cap \Bc_{\Rb}(x_0; r) \Big) =  \dist^{\Haus}\Big( C_n, \emptyset \Big) = \infty
 $$
 for all $n \geq 1$.  
\end{example*}

The next observation makes this naive characterization of convergence precise. 

\begin{observation}
\label{obs:convergence_in_loc_haus_top} 
Suppose $(X,\dist)$ is a proper metric space and $(C_n)_{n \geq 1}$ is a sequence of closed sets in $X$. Then, $C_n \rightarrow C$ in the local Hausdorff convergence topology if and only if 
$$
\lim_{n \rightarrow \infty} \dist^{\Haus}\Big( C_n \cap \Bc_X(x_0;r), C \cap \Bc_X(x_0; r) \Big) =0
$$
for all $x_0 \in X$ and $r > 0$ where $C \cap \Bc_X(x_0, r) \neq \emptyset$. 
\end{observation} 

\begin{proof} $(\Leftarrow)$: Fix $x_0 \in X$ and $r > 0$ such that $C \cap \Bc_X(x_0, r) \neq \emptyset$. Then fix $\epsilon > 0$. Since 
$$
C \in U(C, x_0, r, \epsilon),
$$
there exists $N \geq 1$ such that $C_n \in U(C, x_0, r, \epsilon)$ for all $n \geq N$. Then 
$$
\limsup_{n \rightarrow \infty} \dist^{\Haus}\Big( C_n \cap \Bc_X(x_0;r), C \cap \Bc_X(x_0; r) \Big)\leq \epsilon
$$
by the definition of $U(C,x_0, r, \epsilon)$. Since $\epsilon > 0$ was arbitrary, we see that 
$$
\lim_{n \rightarrow \infty} \dist^{\Haus}\Big( C_n \cap \Bc_X(x_0;r), C \cap \Bc_X(x_0; r) \Big) =0.
$$
$(\Rightarrow)$: Fix an open set $\Uc$, in the local Hausdorff convergence topology, that  contains $C$. Then by the definition of the topology,  there exist $x_0 \in X$ and $r_0, \epsilon_0 > 0$ such that 
$$
C \in U(C, x_0, r_0, \epsilon_0) \subset \Uc. 
$$
In particular, $C \cap \Bc_X(x_0; r_0) \neq \emptyset$. Thus, by hypothesis, 
$$
\lim_{n \rightarrow \infty} \dist^{\Haus}\Big( C_n \cap \Bc_X(x_0;r_0), C \cap \Bc_X(x_0; r_0) \Big) =0.
$$
Then for $n$ sufficiently large, we have
$$
C_n \in U(C, x_0, r_0, \epsilon_0) \subset \Uc.
$$
Thus $C_n \rightarrow C$.
\end{proof} 

As a corollary to this observation, we have the following. 

\begin{corollary}\label{cor:convergence in the local Haus top ii} Suppose $(X,\dist)$ is a proper metric space and $C_n \rightarrow C$ in the local Hausdorff convergence topology. If $p \in X$, then the following are equivalent: 
\begin{enumerate}
\item $p \in C$, 
\item there exists a sequence $(p_n)_{n \geq 1}$ in $X$ such that $p_n \in C_n$ for all $n$ and $p_n \rightarrow p$. 
\end{enumerate}
\end{corollary} 

\begin{proof} Fix $r > 0$ such that $C \cap \Bc_X(p, r) \neq \emptyset$. Then by Observation~\ref{obs:convergence_in_loc_haus_top}, 
$$
\lim_{n \rightarrow \infty} \dist^{\Haus}\Big( C_n \cap \Bc_X(p;r), C \cap \Bc_X(p; r) \Big) =0,
$$
which implies the desired equivalence. 
\end{proof}

Besides the properties mentioned above, the only other property of the local Hausdorff convergence topology we will use in this paper is the following. 

\begin{proposition}\label{prop:PES_closed}Suppose $\Omega \subset \Pb(\Rb^d)$ is a properly convex domain. Then the set of properly embedded simplices in $\Omega$ of dimension at least two is closed in the local Hausdorff convergence topology induced by $\dist_\Omega$. 
\end{proposition}

\begin{proof} This follows from Observation 3.20 in~\cite{IZ2019b}, but we provide a proof for the reader's convenience. 

Suppose $(S_n)_{n \geq 1}$ is a sequence of  properly embedded simplices in $\Omega$ of dimension at least two which converges to a closed subset $S$ in the local Hausdorff convergence topology induced by $\dist_\Omega$. Passing to a subsequence we can suppose that $\dim S_n = k$ for all $n$. 

Let $v_{1}^{(n)}, \dots, v_{k}^{(n)}$ be the vertices of $S_n$. Passing to a subsequence we can suppose that $v_{j}^{(n)} \rightarrow v_j$ for all $j$. To show that $S$ is a properly embedded simplex of dimension $k$ it suffices to show that 
\begin{enumerate}[(a)]
\item\label{item:dumb proof a} $v_1, \dots, v_k$ are linearly independent, 
\item\label{item:dumb proof b} $S = \relint {\rm ConvHull}_{\overline{\Omega}}(\{ v_1, \dots, v_k\})$, and
\item\label{item:dumb proof c} $\Omega \cap \Pb(\Spanset\{v_1,\dots, v_{j-1}, v_{j+1}, \dots, v_k\}) = \emptyset$ for all $j=1,\dots,k$.
\end{enumerate}

First we verify \ref{item:dumb proof c}. Since each $S_n$ is a properly embedded simplex, 
$$
\Omega \cap \Pb\left(\Spanset\left\{v_1^{(n)},\dots, v_{j-1}^{(n)}, v_{j+1}^{(n)}, \dots, v_k^{(n)}\right\}\right) = \emptyset
$$
for all $j =1,\dots, k$ and $n \geq 1$. So sending $n \rightarrow \infty$ and using the fact that $\Omega$ is open, we see that 
\begin{equation}\label{eqn:Omega does intersect the partial spans}
\Omega \cap \Pb(\Spanset\{v_1,\dots, v_{j-1}, v_{j+1}, \dots, v_k\}) = \emptyset
\end{equation}
for all $j=1,\dots,k$. This verifies~\ref{item:dumb proof c}.

Since each $S_n$ is a properly embedded simplex, 
$$
S_n = \relint {\rm ConvHull}_{\overline{\Omega}}\left(\left\{ v_1^{(n)}, \dots, v_k^{(n)}\right\}\right).
$$
So taking limits and using Corollary~\ref{cor:convergence in the local Haus top ii} we see that 
\begin{equation}\label{eqn:S is in the convex hull of the limit points}
S = \Omega\cap {\rm ConvHull}_{\overline{\Omega}}(\{ v_1, \dots, v_k\}).
\end{equation} 

Next we verify \ref{item:dumb proof a}. Suppose $v_1, \dots, v_k$ are not linearly independent. Then $v_j \in \Pb(\Spanset\{v_1,\dots, v_{j-1}, v_{j+1}, \dots, v_k\})$ for some $j$. Then using Equations~\eqref{eqn:Omega does intersect the partial spans} and~\eqref{eqn:S is in the convex hull of the limit points} we have 
$$
S \subset {\rm ConvHull}_{\overline{\Omega}}(\{ v_1, \dots, v_k\}) \subset \Pb(\Spanset\{v_1,\dots, v_{j-1}, v_{j+1}, \dots, v_k\}) \subset \Pb(\Rb^d) \setminus \Omega,
$$
which is a contradiction. Thus ~\ref{item:dumb proof a} is true. 

Finally we verify \ref{item:dumb proof b}. By Equation~\eqref{eqn:S is in the convex hull of the limit points}, it suffices to show that 
$$
\relint {\rm ConvHull}_{\overline{\Omega}}\left(\left\{ v_1, \dots, v_k\right\}\right) \subset \Omega. 
$$
Suppose not. Then there exists
$$
x \in \Big(\relint {\rm ConvHull}_{\overline{\Omega}}\left(\left\{ v_1, \dots, v_k\right\}\right)\Big) \setminus \Omega. 
$$
Since $\Omega$ is convex, there exists a projective hyperplane $H$ such that $x \in H$ and $H \cap \Omega = \emptyset$. Equation~\eqref{eqn:S is in the convex hull of the limit points} implies that $H$ intersects $\Pb(\Spanset\{v_1,\dots, v_k\})$ transversally, i.e. $\Pb(\Spanset\{v_1,\dots,v_k\}) \not\subset H$ (otherwise $S \subset \Pb(\Spanset\{v_1,\dots, v_k\}) \subset H \subset \Pb(\Rb^d)-\Omega$, a contradiction). 

On the other hand, $v_{j}^{(n)} \rightarrow v_j$, the lines $v_1, \dots, v_k$ are linearly independent, and $x \in \relint {\rm ConvHull}_{\overline{\Omega}}\{v_1,\dots,v_k\} \cap H.$ Thus the hyperplane $H$ non-trivially  intersects 
$$
S_n = \relint {\rm ConvHull}_{\overline{\Omega}}\left(\left\{ v_1^{(n)}, \dots, v_k^{(n)}\right\}\right)
$$
for $n$ large. This is impossible since $S_n \subset \Omega$ and thus  \ref{item:dumb proof b} is true. 
  \end{proof}

\subsection{Properly embedded simplices}

In this section we record some basic facts about properly embedded simplices in a properly convex domain. 

The following result is a simple consequence of any of the explicit formulas for the Hilbert metric on a simplex, see~\cite[Proposition 1.7]{N1988}, ~\cite{dlH1993}, or ~\cite{V2014}. 

\begin{proposition}\label{prop:quasi_isometric_simplex} If $\Omega \subset \Pb(\Rb^d)$ is a properly convex domain and $S \subset \Omega$ is a properly embedded simplex, then $(S, \dist_\Omega)$ is quasi-isometric to $\Rb^{\dim S}$. 
\end{proposition}

The faces of a properly embedded simplex are themselves properly embedded simplices in the boundary faces that contain them.

\begin{observation}\label{obs:faces_of_simplices_are_properly_embedded}
Suppose $\Omega \subset \Pb(\Rb^d)$ is a properly convex domain and $S \subset \Omega$ is a properly embedded simplex. If $x \in \partial S$, then 
\begin{enumerate}
\item $F_S(x)$ is a properly embedded simplex in $F_\Omega(x)$. 
\item $F_S(x) = \overline{S} \cap F_\Omega(x)$. 
\end{enumerate}
\end{observation}

\begin{proof} See for instance Observation 5.4 in~\cite{IZ2019b}. \end{proof}

\begin{definition}\label{defn:parallel} Suppose $\Omega \subset \Pb(\Rb^d)$ is a properly convex domain. Two properly embedded simplices $S_1, S_2 \subset \Omega$ are called \emph{parallel} if $\dim S_1 = \dim S_2$ and there is a labeling  $v_1,\dots, v_p$ of the vertices of $S_1$ and a labeling $w_1,\dots, w_p$ of the vertices of $S_2$ such that $F_\Omega(v_k) = F_\Omega(w_k)$ for all $1 \leq k \leq p$. 
\end{definition} 

The following lemma allows us to ``wiggle'' the vertices of a properly embedded simplex and obtain a new parallel properly embedded simplex. 

\begin{lemma}\label{lem:slide_along_faces} Suppose $\Omega \subset \Pb(\Rb^d)$ is a properly convex domain and $S \subset \Omega$ is a properly embedded simplex with vertices $v_1,\dots, v_p$. If $w_j \in F_\Omega(v_j)$ for $1 \leq j \leq p$, then 
\begin{align*}
S^\prime : = \Omega \cap \Pb(\Spanset\{ w_1,\dots, w_p\})=\relint {\rm ConvHull}_{\overline{\Omega}}(w_1,\dots,w_p)
\end{align*}
is a properly embedded simplex with vertices $w_1,\dots,w_p$. Moreover, 
\begin{align*}
\dist_\Omega^{\Haus}\left(S, S^\prime\right) \leq \max_{1 \leq j \leq p} \dist_{F_\Omega(v_j)}(v_j, w_j). 
\end{align*}
\end{lemma}

\begin{proof} See for instance Lemma 3.18 in~\cite{IZ2019b}. \end{proof}

\section{Relative hyperbolic convex co-compact groups}

In this section we recall some properties of general relatively hyperbolic spaces/groups and also recall some of the results from~\cite{IZ2019b}.

\subsection{General relatively hyperbolic groups}
We define relative hyperbolic spaces and groups in terms of Dru{\c t}u and Sapir's tree-graded spaces (see~\cite[Definition 2.1]{DS2005}). 

\begin{definition} \ \begin{enumerate}
\item A complete geodesic metric space $(X,\dist)$ is \emph{relatively hyperbolic with respect to a collection of subsets $\Sc$} if all its asymptotic cones, with respect to a fixed non-principal ultrafilter, are tree-graded with respect to the collection of ultralimits of the elements of $\Sc$. 
\item A finitely generated group $G$ is \emph{relatively hyperbolic with respect to a family of subgroups $\{H_1,\dots, H_k\}$} if the Cayley graph of $G$ with respect to some (hence any) finite set of generators is relatively hyperbolic with respect to the collection of left cosets $\{g H_i : g \in G, i=1,\dots,k\}$. 
\end{enumerate}
\end{definition}

\begin{remark}These are one among several equivalent definitions of relatively hyperbolic spaces/groups, see~\cite{DS2005} and the references therein for more details. 

\end{remark}

If $(X,\dist)$ is a metric space, we will use the following notation for metric tubular neighborhoods: if $A \subset X$ and $r> 0$, then 
\begin{align*}
\Nc_X(A;r):=\left\{ x \in X : \dist(x,A)< r\right\}.
\end{align*}

We will frequently use the following property of relatively hyperbolic spaces. 

\begin{theorem}[{Dru{\c t}u--Sapir~\cite[Corollary 5.8]{DS2005}}]\label{thm:rh_embeddings_of_flats} Suppose $(X, \dist)$ is relatively hyperbolic with respect to $\Sc$. Then for any $A \geq 1$ and $B \geq 0$ there exists $M =M(A,B)$ such that: if $k \geq 2$ and $f : \Rb^k \rightarrow X$ is an $(A,B)$-quasi-isometric embedding, then there exists some $S \in \Sc$ such that 
\begin{align*}
f(\Rb^k) \subset \Nc_X(S;M).
\end{align*}
\end{theorem} 

\subsection{Convex co-compact relatively hyperbolic groups} Next we recall some of the results in~\cite{IZ2019b} describing the structure of (naive) convex co-compact groups which are relatively hyperbolic with respect to a collection of virtually Abelian subgroups of rank at least two. 

In the convex co-compact case we have the following characterization and structural results. 
 
\begin{theorem}[{\cite[Theorem 1.7 and 1.8]{IZ2019b}}]
\label{thm:structure_rel_hyp_cc} Suppose $\Omega \subset \Pb(\Rb^d)$ is a properly convex domain, $\Gamma \leq \Aut(\Omega)$ is convex co-compact, and $\Sc_{\simplexcc}$ is the family of all maximal properly embedded simplices in $\Cc_{\Omega}(\Gamma)$ with dimension at least two. Then the following are equivalent: 
\begin{enumerate}
\item $\Sc_{\simplexcc}$ is closed and discrete in the local Hausdorff convergence topology induced by $\dist_\Omega$,
\item $\Gamma$ is a relatively hyperbolic group with respect to a collection of virtually Abelian subgroups of rank at least two,
\item $(\Cc_\Omega(\Gamma), \dist_\Omega)$ is a relatively hyperbolic space with respect to $\Sc_{\simplexcc}$,
\item $(\Cc_\Omega(\Gamma), \dist_\Omega)$  is relatively hyperbolic with respect to a collection of properly embedded simplices of dimension at least two. 
\end{enumerate}
Moreover, when $\Sc_{\simplexcc}$ is closed and discrete in the local Hausdorff convergence topology induced by $\dist_\Omega$, then:
\begin{enumerate}[(a)]
\item $\Gamma$ has finitely many orbits in $\Sc_{\simplexcc}$. 
\item\label{item:cc_three} If $S \in \Sc_{\simplexcc}$, then $\Stab_{\Gamma}(S)$ acts co-compactly on $S$ and contains a finite index subgroup isomorphic to $\Zb^k$ where $k = \dim S$. 
\item\label{item:cc_four} If $\{S_1,\dots, S_m\}$ are representatives of the $\Gamma$-orbits in $\Sc_{\simplexcc}$, then $\Gamma$ is a relatively hyperbolic group with respect to $\{ \Stab_{\Gamma}(S_1), \dots, \Stab_{\Gamma}(S_m)\}$.
\item If $A \leq \Gamma$ is an infinite Abelian subgroup with rank at least two, then there exists a unique $S \in \Sc_{\simplexcc}$ with $A \leq \Stab_{\Gamma}(S)$. 
\item\label{item:cc_two} If $S \in \Sc_{\simplexcc}$ and $x \in \partial S$, then $F_{\Omega}(x) = F_S(x)$. 
\item If $S_1, S_2 \in \Sc_{\simplexcc}$ are distinct, then $\#(S_1 \cap S_2) \leq 1$ and $\partial S_1 \cap \partial S_2 = \emptyset$. 
\item For any $r > 0$ there exists $D(r) > 0$ such that: if $S_1, S_2 \in \Sc_{\simplexcc}$ are distinct, then 
\begin{align*}
\diam_\Omega \Big( \Nc_\Omega(S_1; r) \cap \Nc_\Omega(S_2; r) \Big) \leq D(r).
\end{align*}
\item\label{item:cc_one} If $\ell \subset \partiali \Cc_{\Omega}(\Gamma)$ is a non-trivial line segment, then there exists $S \in \Sc_{\simplexcc}$ with $\ell \subset \partial S$. 
\end{enumerate}
\end{theorem}

In the naive convex co-compact case, we established a similar characterization and structural results. However, they are much more technical. The main issue is that there can exist bounded families of parallel properly embedded simplices, see for instance \cite[Section 2.3]{IZ2019b}. So the group being relative hyperbolic with respect to a family  of virtually Abelian subgroups of rank at least two is not equivalent to the set of \textbf{all} properly embedded simplices being closed and discrete. Instead, it is equivalent to the existence of a $\Gamma$-invariant family of properly embedded simplices which is closed, discrete, and which coarsely contains every properly embedded simplex. This is made precise in the next definition.

\begin{definition}[{\cite[Definition 1.11]{IZ2019b}}]
\label{defn:IS}
Suppose $(\Omega, \Cc, \Gamma)$ is a naive convex co-compact triple. A family $\Sc$ of maximal properly embedded simplices  in $\Cc$ of dimension at least two is called: 
\begin{enumerate}
\item \emph{Isolated}, if $\Sc$ is closed and discrete in the local Hausdorff convergence topology induced by $\dist_\Omega$.
\item \emph{Coarsely complete}, if any properly embedded simplex in $\Cc$ is contained in a uniformly bounded tubular neighborhood of some properly embedded simplex in $\Sc$.
\item \emph{$\Gamma$-invariant}, if $g \cdot S \in \Sc$ for all $S \in \Sc$ and $g \in \Gamma$.
\end{enumerate}
We say that $(\Omega, \Cc, \Gamma)$ has \emph{coarsely isolated simplices} if there exists an isolated, coarsely complete, and $\Gamma$-invariant family of maximal properly embedded simplices. 
\end{definition}

We then have the following characterization of relative hyperbolicity (with respect to a family  of virtually Abelian subgroups of rank at least two) in the naive convex co-compact case. 

\begin{theorem}[{\cite[Theorem 1.13]{IZ2019b}}]
\label{thm:main_char_ncc} Suppose $(\Omega, \Cc, \Gamma)$ is a naive convex co-compact triple. Then the following are equivalent: 
\begin{enumerate}
\item $(\Omega, \Cc, \Gamma)$ has coarsely isolated simplices,
\item $(\Cc, \dist_\Omega)$ is a relatively hyperbolic space with respect to a family of properly embedded simplices in $\Cc$ of dimension at least two,
\item $\Gamma$ is a relatively hyperbolic group with respect to a family  of virtually Abelian subgroups of rank at least two.
\end{enumerate}
\end{theorem}

The naive convex co-compact case has one more delicate point: if $(\Omega, \Cc, \Gamma)$ is a naive convex co-compact triple and $\Sc$ is a family of properly embedded simplices satisfying Definition~\ref{defn:IS}, then it is not always true that $(\Cc, \dist_\Omega)$ is relatively hyperbolic with respect to $\Sc$, see Observation 2.10 in \cite{IZ2019b} for examples.

Instead one requires an even stricter isolation property: we say that a family of simplices $\Sc$ in a properly convex domain $\Omega$ is \emph{strongly isolated}, if for every $r>0$, there exists $D(r) >0$ such that if $S_1, S_2 \in \Sc$ are distinct, then 
\begin{align*}
\diam_{\Omega}\left( \Nc_{\Omega}(S_1; r) \cap \Nc_{\Omega}(S_2; r)  \right) \leq D(r).
\end{align*}
It is straightforward to see that a \emph{strongly isolated} family of simplices is indeed \emph{isolated}. However, the converse is not true in general (see Section 2.3 - mainly Observation 2.10 - in \cite{IZ2019b}). But we proved in \cite{IZ2019b} that one can modify a coarsely isolated family of simplices to construct a strongly isolated family of simplices.

\begin{theorem}[{\cite[Theorem 1.17]{IZ2019b}}]
\label{thm:ncc_strong_IS}
Suppose $(\Omega, \Cc, \Gamma)$ is a naive convex co-compact triple with coarsely isolated simplices. Then there exists a strongly isolated, coarsely complete, and $\Gamma$-invariant collection of properly embedded simplices in $\Cc$ of dimension at least two.
\end{theorem}

We also proved the following.

\begin{theorem}[{\cite[Theorem 1.18 and 1.19]{IZ2019b}}]
\label{thm:ncc_rel_hyp} Suppose $(\Omega, \Cc, \Gamma)$ is a naive convex co-compact triple with coarsely isolated simplices. If $\Sc$ is a strongly isolated, coarsely complete, and $\Gamma$-invariant collection of properly embedded simplices in $\Cc$ of dimension at least two, then

\begin{enumerate}
\item $(\Cc, \dist_\Omega)$ is relatively hyperbolic with respect to $\Sc$.
\item If $S \in \Sc$, then $\Stab_{\Gamma}(S)$ acts co-compactly on $S$ and contains a finite index subgroup isomorphic to $\Zb^k$ where $k = \dim S$. 
\item $\Gamma$ has finitely many orbits in $\Sc$. 
\item\label{item:ncc_four} If $\{S_1,\dots, S_m\}$ are representatives of the $\Gamma$-orbits in $\Sc$, then $\Gamma$ is a relatively hyperbolic group with respect to $\{ \Stab_{\Gamma}(S_1), \dots, \Stab_{\Gamma}(S_m)\}$.
\item If $A \leq \Gamma$ is an Abelian subgroup with rank at least two, then there exists a unique $S \in \Sc$ with $A \leq \Stab_{\Gamma}(S)$. 
\item\label{item:ncc_six} There exists $D > 0$ such that: If $S \in \Sc$ and $x \in \partial S$, then 
\begin{align*}
\dist_{F_\Omega(x)}^{\Haus}\Big( \overline{\Cc} \cap F_{\Omega}(x), F_S(x) \Big) \leq D.
\end{align*} 
\item If $S_1, S_2 \in \Sc$ are distinct, then $\#(S_1 \cap S_2) \leq 1$ and 
\begin{align*}
\left(\bigcup_{x \in \partial S_1} F_\Omega(x) \right) \bigcap \left( \bigcup_{x \in \partial S_2} F_\Omega(x) \right)= \emptyset.
\end{align*}
\end{enumerate}
\end{theorem}

\section{Properties of coarse dimension}

In this section we make some basic observations about the coarse dimension (see Definition \ref{defn:cdim}).

\begin{observation}\label{obs:qdim_of_faces_PES} Suppose $\Omega \subset \Pb(\Rb^d)$ is a properly convex domain. If $S \subset \Omega$ is a properly embedded simplex, then 
\begin{align*}
\cdim_{F_\Omega(x)} \left( F_S(x) \right) = \dim F_S(x)
\end{align*}
for all $x \in \overline{S}$. 
\end{observation} 

\begin{proof} If $F_S(x)$ is a point, then there is nothing to prove. So we can assume $\dim F_S(x) = k > 0$. Then $F_S(x)$ is a properly embedded simplex in $F_\Omega(x)$ by Observation~\ref{obs:faces_of_simplices_are_properly_embedded}.

Suppose $D \subset F_{\Omega}(x)$ is a convex subsets with 
\begin{align*}
F_S(x) \subset \Nc_{F_\Omega(x)}(D;R)
\end{align*}
for some $R > 0$. Let $v_1,\dots, v_{k+1} \in \partial F_S(x)$ denote the vertices of $S$ in $\overline{F_S(x)}$.  Then by Proposition~\ref{prop:dist_est_and_faces} for each $j \in \{1,\dots, k+1\}$ there exists 
\begin{align*}
w_j \in F_\Omega(v_j) \cap \overline{D}. 
\end{align*}
By Lemma~\ref{lem:slide_along_faces}
\begin{align*} 
S^\prime : = \relint {\rm ConvHull}_{\overline{\Omega}}( w_1,\dots, w_{k+1}) \subset D
\end{align*} 
is a $k$-dimensional  properly embedded simplex in $F_{\Omega}(x)$ and so 
\begin{equation*}
\dim D  \geq \dim S^\prime \geq k = \dim F_S(x). \qedhere
\end{equation*}
 \end{proof}

\begin{observation}\label{obs:lower_bd_from_bd} Suppose $\Omega \subset \Pb(\Rb^d)$ is a properly convex domain and $\Cc \subset \Omega$ is a convex subset. If $\partiali\Cc \neq \emptyset$, then
\begin{align*}
\cdim_\Omega( \Cc) \geq 1 + \max_{x \in \partiali\Cc} \, \cdim_{F_\Omega(x)} \left( \partiali\Cc \cap F_\Omega(x) \right).
\end{align*}
\end{observation}

\begin{proof} Suppose $D \subset \Omega$ is a convex subset with $\dim D = \cdim_\Omega( \Cc)$ and 
\begin{align*}
\Cc \subset \Nc_{\Omega}(D;R)
\end{align*}
for some $R > 0$. Fix $x \in \partiali\Cc$ and let $D_x := \overline{D} \cap F_\Omega(x)$. Proposition~\ref{prop:dist_est_and_faces} implies that 
\begin{align*}
\partiali\Cc \cap F_\Omega(x)  \subset \Nc_{F_\Omega(x)}(D_x; R+1)
\end{align*}
and hence, by definition,  
$$\cdim_{F_{\Omega}(x)} \left( \partiali\Cc \cap F_{\Omega}(x) \right) \leq \dim D_x \leq -1+\dim D=-1+\cdim_\Omega( \Cc).$$

 \end{proof}

\begin{observation}\label{obs:lower_bd_from_three_different_faces} Suppose $\Omega \subset \Pb(\Rb^d)$ is a properly convex domain and $\Cc \subset \Omega$ is a convex subset. If there exist $x_1,x_2,x_3 \in \partiali\Cc$ such that $F_\Omega(x_1), F_\Omega(x_2), F_\Omega(x_3)$ are pairwise distinct, then 
\begin{align*}
\cdim_\Omega( \Cc) \geq 2.
\end{align*}
\end{observation}

\begin{proof} Suppose not. Then there exists a convex subset $D \subset \Omega$ where $\dim D \leq 1$ and
\begin{align*}
\Cc \subset \Nc_{\Omega}(D;R)
\end{align*}
for some $R > 0$. Then by Proposition~\ref{prop:dist_est_and_faces} for each $j \in \{1,2,3\}$ there exists 
\begin{align*}
y_j \in F_\Omega(x_j) \cap \overline{D} \subset \partiali D
\end{align*}
By assumption $y_1,y_2,y_3$ are pairwise distinct. However, since $\dim D \leq 1$ the set $ \partiali D$ contains at most two points. So we have a contradiction. 

\end{proof}

Next we show that a certain configuration of points in the ideal boundary of a naive convex co-compact triple implies that the boundary contains a face with coarse dimension at least two. 

\begin{proposition}\label{prop:finding_two_d_faces} Suppose $(\Omega, \Cc, \Gamma)$ is a naive convex co-compact triple. If there exist distinct points $x, y_1,y_2, y_3 \in \partiali\Cc$ such that 
\begin{align*}
 [x, y_1]\cup [x,y_2] \cup [x,y_3] \subset \partial \Omega
\end{align*}
and 
\begin{align*}
(y_1,y_2) \cup (y_2,y_3) \cup (y_3, y_1) \subset \Omega,
\end{align*}
then there exists $w \in \partiali\Cc$ with 
\begin{align*}
\cdim_{F_\Omega(w)} \Big( \partiali \Cc \cap F_\Omega(w) \Big) \geq 2.
\end{align*}
\end{proposition}

\begin{proof} Fix $u \in \relint {\rm ConvHull}_{\overline{\Omega}}(y_1,y_2, y_3)$, $a \in (y_1,y_2)$, $b \in (y_2, y_3)$, and $c \in (y_3, y_1)$. Then we can find sequences $u_n \in (x,u)$, $a_n \in (x,a)$, $b_n \in (x,b)$, and $c_n \in (x,c)$ all converging to $x$ such that 
\begin{align*}
\dist_\Omega(u_n,a_n) \leq \dist_\Omega(u, a), \ \dist_\Omega(u_n,b_n) \leq \dist_\Omega(u, b), \text{ and }  \dist_\Omega(u_n,c_n) \leq \dist_\Omega(u, c)
\end{align*}
for all $n$. 

By passing to a subsequence we can find $\gamma_n \in \Gamma$ such that $\gamma_n u_n \rightarrow \hat{u} \in \Cc$ and 
\begin{align*}
\gamma_n a_n, \gamma_n b_n, \gamma_n c_n, \gamma_n x, \gamma_n y_1, \gamma_n y_2, \gamma_n y_3 \rightarrow \hat{a}, \hat{b},\hat{c},\hat{x}, \hat{y}_1, \hat{y}_2, \hat{y}_3.
\end{align*}
Then 
\begin{align*}
[\hat{x}, \hat{y}_1] \cup [\hat{x}, \hat{y}_2] \cup [\hat{x}, \hat{y}_3]  \subset \partiali\Cc
\end{align*}
and by our choice of sequences $\hat{a}, \hat{b},\hat{c} \in \Cc$. Also, since $u_n \rightarrow x \in \partial \Omega$, we have
\begin{align*}
\lim_{n \rightarrow \infty} \dist_\Omega\Big( u_n, \Omega \cap {\rm ConvHull}_{\overline{\Omega}}(y_1,y_2,y_3) \Big) = \infty
\end{align*}
and so 
\begin{align*}
{\rm ConvHull}_{\overline{\Omega}}\left(\hat{y}_1, \hat{y}_2, \hat{y}_3\right) \subset \partiali\Cc.
\end{align*}
 Fix $w \in \relint {\rm ConvHull}_{\overline{\Omega}}\left(\hat{y}_1, \hat{y}_2, \hat{y}_3\right) \subset \partiali\Cc$. 
 
  \medskip
\noindent \fbox{Claim 1:} $(\hat{x}, w) \subset \Omega$. 
  \medskip
 
 Since 
 \begin{align*}
 \hat{u} \in \Omega \cap {\rm ConvHull}_{\overline{\Omega}}\left(\hat{x}, \hat{y}_1, \hat{y}_2, \hat{y}_3\right)
 \end{align*}
 convexity implies that $(\hat{x}, w) \subset \Omega$. 
 
 \medskip
\noindent \fbox{Claim 2:} $\hat{y}_1, \hat{y}_2, \hat{y}_3 \in \partial F_\Omega(w)$. 
  \medskip
  
By construction, $\hat{y}_1, \hat{y}_2, \hat{y}_3 \in \overline{F_\Omega(w)}$. Fix $ j \in \{1, 2, 3\}$. Since $(\hat{x}, w) \subset \Omega$ and $[\hat{x}, \hat{y}_j] \subset \partial \Omega$, Observation~\ref{obs:faces} part (4) implies that $\hat{y}_j \notin F_\Omega(w)$. So $\hat{y}_j  \in  \partial F_\Omega(w)$. 

\medskip 
\noindent \fbox{Claim 3:}  $F_\Omega(\hat{y}_1),F_\Omega(\hat{y}_2),F_\Omega(\hat{y}_3)$ are pairwise distinct. 
 \medskip 

By symmetry it is enough to show that $F_\Omega(\hat{y}_1)$ and $F_\Omega(\hat{y}_2)$ are distinct. If not, then 
 \begin{align*}
[\hat{y}_1, \hat{y}_2] \subset F_\Omega(\hat{y}_2).
 \end{align*}
 Since $[\hat{x}, \hat{y}_1] \subset \partial \Omega$, Observation~\ref{obs:faces} part (4) then implies that 
 \begin{align*}
 {\rm ConvHull}_{\overline{\Omega}}\left(\hat{x}, \hat{y}_1, \hat{y}_2\right) \subset \partial\Omega. 
 \end{align*}
 However $\hat{a} \in \Cc \subset \Omega$ is contained in this convex hull and hence we have a contradiction. 

 \medskip 
\noindent \fbox{Claim 4:} $\cdim_{F_\Omega(w)} \Big( \partiali \Cc \cap F_\Omega(w) \Big) \geq 2$.
\medskip

This follows immediately from Claim 2, Claim 3, and Observation~\ref{obs:lower_bd_from_three_different_faces}.
\end{proof}

\section{Proof of Theorem~\ref{thm:ncc_GH}}\label{sec:pf_of_GH_ncc}

In this section we prove the following extension of Theorem~\ref{thm:ncc_GH}.

\begin{theorem}\label{thm:ncc_GH_2} Suppose $(\Omega, \Cc, \Gamma)$ is a naive convex co-compact triple. Then the following are equivalent: 
\begin{enumerate}
\item there exists $R > 0$ such that $\diam_{F_\Omega(x)}\left( \partiali\Cc \cap F_\Omega(x) \right)\leq R$ for all $x \in \partiali\Cc$,  
\item $\cdim_{F_\Omega(x)} \left( \partiali\Cc \cap F_\Omega(x) \right)=0$ for all $x \in \partiali\Cc$,  
\item $\Cc$ does not contain a properly embedded simplex with dimension at least two, 
\item $(\Cc, \dist_\Omega)$ is Gromov hyperbolic, 
\item $\Gamma$ is a word hyperbolic group.
\end{enumerate}
\end{theorem}

By definition $(1) \Rightarrow (2)$, by Observation~\ref{obs:qdim_of_faces_PES} $(2) \Rightarrow (3)$, by Proposition~\ref{prop:quasi_isometric_simplex} $(4) \Rightarrow (3)$, and by the {\v S}varc-Milnor Lemma $(4) \Leftrightarrow (5)$. We will complete the proof by showing that $(3) \Rightarrow (1)$ and $(2) \Rightarrow (4)$.

In the convex co-compact case, it is well known that a line segment in the ideal boundary implies the existence of a properly embedded simplex. This is given explicitly in~\cite[Lemma 6.2]{DGF2017} using a proof nearly identical to~\cite[Proposition 2.5]{B2004} and \cite[Lemma 3.9]{B2006}. Unfortunately, simple examples show that this observation fails in the naive co-convex co-compact case (see Section 2.3 in~\cite{IZ2019b}). The next lemma uses Benoist's argument to establish a more technical condition to guarantee that the existence of a properly embedded simplex.

\begin{lemma}[$(3) \Rightarrow (1)$]\label{lem:existence_of_2_simplices} If $(\Omega, \Cc, \Gamma)$ is a naive convex co-compact triple and
\begin{equation}
\label{eqn:condition in the proof that (3) => (1)} 
\sup_{x \in \partiali\Cc} \diam_{F_\Omega(x)}\left( \partiali \Cc \cap F_\Omega(x) \right) =+\infty,
\end{equation}
then $\Cc$ contains a properly embedded two dimensional simplex. 
\end{lemma} 

\begin{remark} In the convex co-compact case, one can show that if $\partiali \Cc \cap F_\Omega(x) \neq \emptyset$, then  $F_\Omega(x) \subset \partiali\Cc$; see for instance \cite[Section 4]{DGF2017}. So in this special case, Equation~\eqref{eqn:condition in the proof that (3) => (1)} is equivalent to the condition that $\partiali\Cc$ contains a line segment. \end{remark}

\begin{proof} Fix a sequence $(x_n)_{n \geq 1}$ in $\partiali\Cc$ such that 
\begin{align*}
 \diam_{F_\Omega(x_n)}\left( \partiali \Cc \cap F_\Omega(x_n) \right) > n
 \end{align*}
 for all $n$. Then fix $a_n,b_n \in \partiali\Cc \cap F_\Omega(x_n)$ with 
 \begin{align*}
 \dist_{F_\Omega(x_n)}(a_n,b_n) > n.
 \end{align*}
 We can assume that $x_n$ is the midpoint of $[a_n, b_n]$ relative to the Hilbert distance $\dist_{F_{\Omega}(x_n)}$. Also fix some $p_0 \in \Cc$. 
 
 \medskip 
 
 \noindent \fbox{Claim:} For each $n$ there exists $y_n \in [p_0, x_n) \subset \Cc$ such that 
 \begin{align*}
\min\left\{ \dist_\Omega\Big(y_n, [p_0, a_n)\Big), \dist_\Omega\Big(y_n, [p_0, b_n)\Big) \right\}> n/2.
 \end{align*}
 
 Fix $n$ and suppose not. Then we can find $x_{n,m} \in [p_0, x_n)$, $a_{n,m} \in [p_0, a_n)$, and $b_{n,m} \in [p_0, b_n)$ such that $\lim_{m \rightarrow \infty} x_{n,m} = x_n$ and 
  \begin{align*}
\dist_\Omega\Big(x_{n,m},\{a_{n,m}, b_{n,m}\}\Big) \leq  n/2 \text{ for all $m$}. 
 \end{align*}
 By passing to a subsequence and possibly relabelling $a_n,b_n$ we can assume that 
   \begin{align*}
\dist_\Omega\Big(x_{n,m},\{a_{n,m}, b_{n,m}\}\Big)=\dist_\Omega\Big(x_{n,m},a_{n,m} \Big) \leq  n/2 \text{ for all $m$}. 
 \end{align*}
Then we must have $\lim_{m \rightarrow \infty} a_{n,m} = a_n$ and by Proposition \ref{prop:dist_est_and_faces} 
   \begin{align*}
n/2 \geq \limsup_{m \rightarrow \infty} \dist_\Omega(x_{n,m},a_{n,m}) \geq \dist_{F_\Omega(x_n)}(x_{n},a_{n}) > n/2.
 \end{align*}
 So we have a contradiction and hence the claim is established.

 Next let $(\gamma_n)_{n \geq 1}$ be a sequence in $\Gamma$ such that $\{ \gamma_n y_n : n \geq 1\}$ is relatively compact in $\Cc$. By passing to subsequences we can suppose that 
 \begin{align*}
 \gamma_ny_n, \gamma_n a_n, \gamma_n b_n, \gamma_n p_0 \rightarrow y, a, b, p \in \overline{\Cc}.
 \end{align*}
 Then $y \in \Cc$, by construction $[a,b] \subset \partiali\Cc$, and by the claim 
 \begin{align*}
[b,p] \cup [p,a] \subset \partiali\Cc.
 \end{align*}
 So $a,b,p$ are the vertices of a properly embedded simplex $S \subset \Cc$ which contains $y$. 

\end{proof}

To show that $(2) \Rightarrow (4)$ we will use the following sufficient condition for a metric to be Gromov hyperbolic. 

\begin{proposition}\label{prop:GH_suff} Suppose $(X,\dist)$ is a proper geodesic metric space, $\delta > 0$, and there exists a map 
\begin{align*}
(x,y) \in X \times X \mapsto \sigma_{x,y}  \in C\Big([0,\dist(x,y)],X\Big)
\end{align*}
where $\sigma_{x,y}$ is a geodesic joining $x$ to $y$. If for every $x,y,z \in X$ distinct, the geodesic triangle formed by $\sigma_{x,y}, \sigma_{y,z}, \sigma_{z,x}$ is $\delta$-thin, then $(X,\dist)$ is Gromov hyperbolic. 
\end{proposition}

\begin{proof}
This proposition is a straightforward and well known consequence of the  Gromov product definition of Gromov hyperbolicity, see for instance~\cite[Proposition 2.2]{Z2017} for a detailed proof. 
\end{proof}

\begin{lemma}[$(2) \Rightarrow (4)$] Suppose $(\Omega, \Cc, \Gamma)$ is a naive convex co-compact triple. If 
\begin{align*}
\cdim_{F_\Omega(x)} \left( \partiali\Cc \cap F_\Omega(x) \right)=0
\end{align*}
for all $x \in \partiali\Cc$, then $(\Cc, \dist_\Omega)$ is Gromov hyperbolic.
\end{lemma}

\begin{proof}  By Proposition~\ref{prop:GH_suff} it suffices to show that there exists $\delta > 0$ such that every geodesic triangle in $(\Cc, \dist_\Omega)$ whose sides are line segments is $\delta$-thin. Suppose not. Then for every $n \geq 0$  there exist $a_n, b_n, c_n \in \Cc$ and $u_n \in [a_n, b_n]$ such that 
\begin{align}
\label{eqn:not_thin_tri}
\dist_\Omega\left( u_n, [a_n, c_n] \cup [c_n, b_n]\right) > n.
\end{align}
By translating by $\Gamma$ and passing to a subsequence we can suppose that $u_n \rightarrow u \in \Cc$ and 
\begin{align*}
a_n, b_n, c_n \rightarrow a,b,c \in \overline{\Cc}. 
\end{align*}
By Equation~\eqref{eqn:not_thin_tri} we must have 
\begin{align*}
[a,c] \cup [c,b] \subset \partiali\Cc
\end{align*}
and by construction we have $u \in [a,b]$. Then $(a,b) \subset \Omega$ since $u \in \Omega$. 

Since $[a,c] \cup [c,b] \subset \partial \Omega$ and $(a,b) \subset \Omega$, Observation~\ref{obs:faces} part (4) implies that $c \in \partial F_\Omega(a)$. Then Observation~\ref{obs:lower_bd_from_bd}  implies that 
\begin{align*}
\cdim_{F_\Omega(a)} \left( \partiali\Cc \cap F_\Omega(a) \right) \geq 1
\end{align*}
and we have a contradiction. 
\end{proof}

\section{Proof of Theorem~\ref{thm:cc_one_d_faces}}\label{sec:pf_of_ond_d_faces_cc}

In this section we prove Theorem~\ref{thm:cc_one_d_faces} which we restate here. 

\begin{theorem}  Suppose $\Omega \subset \Pb(\Rb^d)$ is a properly convex domain, $\Gamma \subset \Aut(\Omega)$ is convex co-compact, and $\Cc : = \Cc_\Omega(\Gamma)$. Then the following are equivalent: 
\begin{enumerate}
\item every boundary face of $\Omega$ which intersects $\overline{\Cc}$ has dimension at most one,
\item the collection of all properly embedded simplices in $\Cc$ with dimension two is closed and discrete in the local Hausdorff convergence topology induced by $\dist_\Omega$, 
\item $(\Cc,\dist_\Omega)$ is relatively hyperbolic with respect to a collection of two dimensional properly embedded simplices, 
\item $\Gamma$ is a relatively hyperbolic group with respect to a collection of virtually Abelian subgroups of rank two.
\end{enumerate}
\end{theorem}

For the rest of the section suppose that $\Omega \subset \Pb(\Rb^d)$ is a properly convex domain, $\Gamma \subset \Aut(\Omega)$ is convex co-compact, and $\Cc : = \Cc_\Omega(\Gamma)$. 

The implications $(2) \Rightarrow (3) \Rightarrow (4) \Rightarrow (1)$ are easy applications of Theorem~\ref{thm:structure_rel_hyp_cc}.

\subsection{(2) $\Rightarrow$ (3)} Suppose that the collection of all properly embedded simplices in $\Cc$ with dimension two is closed and discrete in the local Hausdorff convergence topology induced by $\dist_\Omega$. 

Then every properly embedded simplex in $\Cc$ has dimension at most two. So the collection of all properly embedded simplices in $\Cc$ with dimension at least two coincides with the collection of all properly embedded simplices in $\Cc$ with dimension two. So  Theorem~\ref{thm:structure_rel_hyp_cc} implies that $(\Cc,\dist_\Omega)$ is relatively hyperbolic with respect to a collection of two dimensional properly embedded simplices.

\subsection{(3) $\Rightarrow$ (4)} Suppose that $(\Cc,\dist_\Omega)$ is relatively hyperbolic with respect to a collection $\Sc$ of two dimensional properly embedded simplices. 

We claim that every properly embedded simplex in $\Cc$ has dimension at most two. Suppose that $S \subset \Cc$ is a properly embedded simplex with dimension at least two. Then $(S,\dist_\Omega)$ is quasi-isometric to $\Rb^{\dim S}$, see Proposition~\ref{prop:quasi_isometric_simplex}. So by Theorem~\ref{thm:rh_embeddings_of_flats} there exist $S^\prime \in \Sc$ and $R >0$ such that $S \subset \Nc_\Omega(S^\prime; R)$. Since $(S^\prime, \dist_\Omega)$ is quasi-isometric to $\Rb^{2}$ we must have $\dim S = 2$. 

Then by Theorem~\ref{thm:structure_rel_hyp_cc} part~\ref{item:cc_four}, $\Gamma$ is a relatively hyperbolic group with respect to a collection of virtually Abelian subgroups of rank two.

\subsection{(4) $\Rightarrow$ (1)} Suppose that $\Gamma$ is a relatively hyperbolic group with respect to $\{H_1,\dots, H_m\}$ where each $H_j$ is a virtually Abelian subgroup of rank two. 

Let $\Sc_{\simplexcc}$ denote the family of all maximal properly embedded simplices in $\Cc_{\Omega}(\Gamma)$ of dimensional at least two. 

Fix $w \in \partiali \Cc$. We will show that $\dim F_\Omega(w) \leq 1$. It suffices to consider the case when $\dim F_\Omega(w) > 0$.  Then Theorem~\ref{thm:structure_rel_hyp_cc}  parts~\ref{item:cc_two} and~\ref{item:cc_one} imply that there exists a simplex $S \in \Sc_{\simplexcc}$ such that $F_\Omega(w) \subset \partial S$. Notice that $\dim S \geq 1+ \dim F_\Omega(w)$ and $(S,\dist_\Omega)$ is quasi-isometric to $\Rb^{\dim S}$, see Proposition~\ref{prop:quasi_isometric_simplex}.

Fix some $p \in \Cc$. By the {\v S}varc-Milnor Lemma and Theorem~\ref{thm:rh_embeddings_of_flats} there exists a coset $gH_j$ such that $S$ is contained in a bounded neighborhood of $gH_j \cdot p$ in $(\Cc, \dist_\Omega)$. Since $H_j$ is virtually isomorphic to $\Zb^2$, we must have $\dim S=2$. Thus 
\begin{align*}
\dim F_\Omega(w)  \leq -1 + \dim S = 1.
\end{align*}

Since $w$ was an arbitrary point in $\partiali\Cc$, every boundary face of $\Omega$ which intersects $\overline{\Cc}$ has dimension at most one.

\subsection{(1) $\Rightarrow$ (2)} Suppose that every boundary face of $\Omega$ which intersects $\overline{\Cc}$ has dimension at most one.

Then, $\Cc$ does not contain any properly embedded simplices with dimension three or more. Hence, using Theorem~\ref{thm:structure_rel_hyp_cc}, it is enough to show that the collection of all properly embedded two dimensional simplices in $\Cc$ is closed and discrete in the local Hausdorff convergence topology induced by $\dist_\Omega$. 

\begin{lemma}\label{lem:lines_intersecting} If $\ell \subset \partiali \Cc$ is a line segment, $S \subset \Cc$ is a properly embedded two dimensional simplex, and $\ell \cap \partial S \neq \emptyset$, then $\ell \subset \partial S$. \end{lemma}

\begin{proof} Suppose for a contradiction that there exists a line segment $\ell \subset \partiali \Cc$ and a properly embedded two dimensional simplex
$S \subset \Cc$ such that $\ell \cap \partial S \neq \emptyset$, but $\ell$ is not contained in $\partial S$. By replacing $\ell$ with a subinterval we can suppose that  $\ell$ intersects $\partial S$ at a single point $x$. 

If $x$ is in a one dimensional boundary face $F$ of $S$, then the convex hull of $\ell$ and $F$ provides a face in $\partiali\Cc$ with dimension at least two. So $x$ must be a vertex of $S$. 

Let $F_1, F_2 \subset \partial S$ be the edges adjacent to $x$. Then pick $y_1 \in F_1$, $y_2 \in F_2$, and $y_3 \in \relint(\ell)$. Then, $(y_1,y_2) \subset S \subset \Omega$. If we had $[y_1,y_3] \subset \partial \Omega$, then the convex hull of $F_1$ and $\ell$ provides a face in $\partiali\Cc$ with dimension at least two. So $(y_1, y_3) \subset \Omega$. For the same reasons, $(y_2, y_3) \subset \Omega$. Then Proposition~\ref{prop:finding_two_d_faces} implies that there exists a face $\partiali\Cc$ with dimensional at least two. So we have a contradiction. 
\end{proof}

Lemma~\ref{lem:lines_intersecting} has the following consequences. 

\begin{lemma}\label{lem:bd_of_triangles_intersecting} If $S_1, S_2 \subset \Cc$ are distinct properly embedded two dimensional simplices, then $\partial S_1 \cap \partial S_2 = \emptyset$. \end{lemma}

\begin{lemma}\label{lem:faces_of_triangles} If $S \subset \Cc$ is a properly embedded two dimensional simplex, then 
\begin{align*}
\partial S  = \bigcup_{x \in \partial S} F_\Omega(x).
\end{align*}
\end{lemma}

We complete the proof of (1) $\Rightarrow$ (2) by showing the following.

\begin{lemma} 
\label{lem:cc_collection_is_isolated}
The collection of properly embedded two dimensional simplices in $\Cc$ is closed and discrete in the local Hausdorff convergence topology. \end{lemma}

\begin{proof} By Proposition~\ref{prop:PES_closed} the collection of properly embedded two dimensional simplices in $\Cc$ is closed in the local Hausdorff convergence topology. So we just have to verify discreteness. 

Suppose that $S_n \rightarrow S$ in the local Hausdorff convergence topology. We need to show that $S_n = S$ for $n$ sufficiently large. Suppose not, then by passing to a subsequence we can assume that $S_n \neq S$ for all $n$. 

Fix $p_0 \in S$. Then for $n \geq 0$ let 
\begin{align*}
R_n := \sup\left\{ r \geq 0 : S \cap \Bc_\Omega(p_0; r) \subset \overline{\Nc_\Omega(S_n; 1)} \right\}.
\end{align*}
If $R_n = \infty$ for some $n$, then 
\begin{align*}
S\subset \overline{\Nc_\Omega(S_n; 1)}.
\end{align*}
So by Proposition~\ref{prop:dist_est_and_faces} and Lemma~\ref{lem:faces_of_triangles}
\begin{align*}
\partial S \subset \bigcup_{x \in \partial S_n} F_\Omega(x) = \partial S_n.
\end{align*}
So Lemma \ref{lem:bd_of_triangles_intersecting} implies that $S=S_n$. Thus we can assume that $R_n < \infty$ for all $n$. Further, since $S_n \rightarrow S$ in the local Hausdorff convergence topology, we see that $R_n \rightarrow \infty$ (see \cref{obs:convergence_in_loc_haus_top}).

Then there exists a sequence $(q_n)_{n \geq 1}$ in $S$ such that 
\begin{enumerate}
\item $\lim_{n \rightarrow \infty} \dist_\Omega(q_n,p_0) = \infty$, 
\item $[q_n, p_0] \subset  \overline{\Nc_\Omega(S_n; 1)}$, and 
\item $\dist_\Omega(q_n, S_n) = 1$. 
\end{enumerate}

Next pick $\gamma_n \in \Gamma$ such that $\{ \gamma_n q_n : n \geq 0\}$ is a relatively compact set in $\Cc$. Then by passing to a subsequence we can suppose that $\gamma_n q_n \rightarrow q \in \Cc$ and $\gamma_n p_0 \rightarrow p \in \partiali \Cc$. Using Proposition~\ref{prop:PES_closed} and passing to another subsequence we can suppose that $\gamma_n S_n \rightarrow S^\prime$ and $\gamma_n S \rightarrow S^{\prime\prime}$ where $S^\prime$ and $S^{\prime\prime}$ are both properly embedded two dimensional simplices in $\Cc$. Further, 
\begin{align*}
[q,p) \subset S^{\prime\prime} \cap \overline{\Nc_\Omega(S^\prime; 1)}.
\end{align*}
Then Proposition~\ref{prop:dist_est_and_faces} implies that $p \in \partial S^{\prime\prime} \cap \bigcup_{s' \in \partial S'}F_{\Omega}(s')$. Then $p \in \partial S^{\prime\prime} \cap \partial S^\prime$ by Lemma~\ref{lem:faces_of_triangles}. So $S^{\prime\prime} = S^{\prime}$ by Lemma~\ref{lem:bd_of_triangles_intersecting}. However, by construction $q \in S^{\prime\prime}$ and $\dist_\Omega(q,S^\prime) = 1$. So we have a contradiction. 
\end{proof}

\section{Proof of Theorem~\ref{thm:ncc_one_d_faces}}\label{sec:pf_of_ond_d_faces_ncc}

In this section we prove Theorem~\ref{thm:ncc_one_d_faces} which we restate here. 

\begin{theorem}Suppose $(\Omega, \Cc, \Gamma)$ is a naive convex co-compact triple. Then the following are equivalent: 
\begin{enumerate}
\item $\cdim_{F_\Omega(x)} \left(\partiali\Cc \cap F_\Omega(x)\right) \leq 1$ for all $x \in \partiali\Cc$,
\item $(\Cc,\dist_\Omega)$ is relatively hyperbolic with respect to a collection of two dimensional properly embedded simplices, 
\item $\Gamma$ is a relatively hyperbolic group with respect to a collection of virtually Abelian subgroups of rank two.
\end{enumerate}
\end{theorem}

 The proof is similar in structure to the proof of Theorem~\ref{thm:cc_one_d_faces} in the previous section, but extending the argument to the naive convex co-compact case introduces a number of technicalities, especially in the proof that (1) $\Rightarrow$ (2). 

Suppose for the rest of the section that $(\Omega, \Cc, \Gamma)$ is a naive convex co-compact triple. We also recall a notation that will be used frequently below: if $X \subset \overline{\Omega}$ is a subset, then 
\begin{align*}
F_{\Omega}(X)=\bigcup_{x \in X} F_{\Omega}(x).
\end{align*}

\subsection{(2) $\Rightarrow$ (3)} Suppose that $(\Cc,\dist_\Omega)$ is relatively hyperbolic with respect to a collection $\Sc$ of two dimensional properly embedded simplices. 

We claim that every properly embedded simplex in $\Cc$ has dimension at most two. Suppose that $S \subset \Cc$ is a properly embedded simplex with dimension at least two. Then $(S,\dist_\Omega)$ is quasi-isometric to $\Rb^{\dim S}$, see Proposition~\ref{prop:quasi_isometric_simplex}. So by Theorem~\ref{thm:rh_embeddings_of_flats} there exist $S^\prime \in \Sc$ and $R >0$ such that $S \subset \Nc_\Omega(S^\prime; R)$. Since $(S^\prime, \dist_\Omega)$ is quasi-isometric to $\Rb^{2}$ we must have $\dim S \leq 2$. 

Then by Theorem~\ref{thm:ncc_rel_hyp} part~\eqref{item:ncc_four}, $\Gamma$ is a relatively hyperbolic group with respect to a collection of virtually Abelian subgroups of rank two.

\subsection{(3) $\Rightarrow$ (2)} Suppose that $\Gamma$ is a relatively hyperbolic group with respect to $\{H_1,\dots, H_m\}$ where each $H_j$ is a virtually Abelian subgroup of rank two. Then Theorem \ref{thm:main_char_ncc} implies that $(\Cc, \dist_{\Omega})$ is a relatively hyperbolic space with respect to a family $\Sc$ of properly embedded simplices of dimension at least two. Thus it is enough to show that if $S \in \Sc$, then $\dim(S) = 2$. 

Fix $S \in \Sc$. Then Proposition \ref{prop:quasi_isometric_simplex} implies that $(S,\dist_{\Omega})$ is quasi-isometric to $\Rb^{\dim(S)}$. Next fix some $p \in \Cc$. By the {\v S}varc-Milnor Lemma and Theorem~\ref{thm:rh_embeddings_of_flats} there exists a coset $gH_j$ such that $S$ is contained in a bounded neighborhood of $gH_j \cdot p$ in $(\Cc, \dist_\Omega)$. Since $H_j$ is virtually isomorphic to $\Zb^2$, we must have $\dim S=2$.

\subsection{(2) $\Rightarrow$ (1)}

Suppose $(\Cc, \dist_\Omega)$ is relatively hyperbolic with respect to a collection $\Sc$ of two dimensional properly embedded  simplices. By Theorems \ref{thm:main_char_ncc} and \ref{thm:ncc_strong_IS}, there exists a strongly isolated, coarsely complete, and $\Gamma$-invariant collection $\Sc_0$ of properly embedded simplices in $\Cc$ of dimension at least two. By Proposition \ref{prop:quasi_isometric_simplex} and Theorem~\ref{thm:rh_embeddings_of_flats}, each simplex in $\Sc_0$ is contained in a bounded neighborhood of a simplex in $\Sc$. Hence each simplex in $\Sc_0$ is two dimensional. 

Fix $w \in \partiali \Cc$. We will show that $\cdim_{F_{\Omega}(w)}( \partiali \Cc \cap F_{\Omega}(w)) \leq 1$. It suffices to consider the case when $\cdim_{F_{\Omega}(w)}( \partiali \Cc \cap F_{\Omega}(w)) >0$. Then
\begin{align*}
\diam_{F_{\Omega}(w)}(\partiali \Cc \cap F_{\Omega}(w))=+\infty
\end{align*} 
which implies that there exists $$w^\prime \in \partiali \Cc \cap \partial F_{\Omega}(w).$$ We first prove the following lemma showing that if we approach points on $(w,w')$ non-tangentially (i.e. along a projective geodesic ray), then we are close to some properly embedded simplex. This can be viewed as a quantitative version of ~\cite[Proposition 2.5]{B2004} or \cite[Lemma 3.9]{B2006}. 

\begin{lemma}
\label{lem:cdim2-implies-simplex}
For any $r, \varepsilon >0$ and $p \in \Cc$, there exist $w_0\in (w,w^\prime)$ and $p_0 \in [p,w_0)$ such that: if $x \in [p_0,w_0)$, then there exists a properly embedded simplex $S_x$ in $\Cc$ of dimension at least two such that 
\begin{align*}
\Pb \left( \Spanset\{ w,w^\prime,p\} \right) \cap \Bc_\Omega(x;r) \subset \Nc_{\Omega}(S_x;\varepsilon).
\end{align*}
\end{lemma}

\begin{proof} Since $w^\prime \in \partial F_\Omega(w)$, for each $n$ we can find $w_n \in (w,w^\prime)$ such that 
$$
\dist_{F_\Omega(w)}(w,w_n)=n.
$$
Then $w_n \to w^\prime$. Fix  $r, \varepsilon > 0$ and $p \in \Cc$. Suppose that the lemma fails. So in particular, it fails for each $w_n$.  Then for each $n \geq 1$, there exists a sequence $(q_{n,m})_{m \geq 1}$ in $[p,w_n)$ with $\lim_{m \to \infty}q_{n,m}=w_n$ and
\begin{align}
\label{eqn:contradiction-step-for-cdim2}
\Pb \left( \Spanset\{ w, w^\prime, p \} \right) \cap \Bc_{\Omega}(q_{n,m};r) \not \subset \Nc_{\Omega}(S;\varepsilon)
\end{align}
for any properly embedded simplex $S$ in $\Cc$ of dimension at least two. By Proposition \ref{prop:dist_est_and_faces}, 
\begin{align*}
\liminf_{m \to \infty} \dist_{\Omega}(q_{n,m},[p,w) \cup [p,w^\prime)) \geq  \dist_{F_{\Omega}(w)}(w_n,w) = n.
\end{align*}
Then for each $n$, we choose $m_n$ large enough such that 
\begin{align}
\label{eqn:dist-qn-from-edges-is-n}
\dist_{\Omega}\left(q_{n,m_n},[p,w] \cup [p,w^\prime]\right) \geq n/2.
\end{align}
Set $q'_n:=q_{n,m_n}$. 

Since $\Gamma$ acts co-compactly on $\Cc$, we can pass to a subsequence and choose $\gamma_n \in \Gamma$ such that $\gamma_n q'_n \to q'_\infty \in \Cc$. Up to passing to another subsequence, we can assume that
\begin{align*}
\gamma_n w', \gamma_n w, \gamma_n p  \to w'_\infty, w_\infty, p_\infty \in \overline{\Cc}.
\end{align*}  
By construction and by Equation \eqref{eqn:dist-qn-from-edges-is-n}, 
\begin{align*}
[p_\infty,w'_\infty] \cup [w'_\infty,w_\infty] \cup [w_\infty,p_\infty] \subset \partiali \Cc.
\end{align*}
Thus,
\begin{align*}
S:=\relint \left( {\rm ConvHull}_{\overline{\Omega}}\{ w_{\infty}, w'_{\infty}, p_{\infty}\} \right)
\end{align*}
is a properly embedded two dimensional simplex in $\Cc$ which contains $q_\infty^\prime$. 
Then 
\begin{align*}
\Pb \left( \Spanset\{ w, w^\prime, p \} \right) \cap \Bc_{\Omega}(q_{n}^\prime;r) \subset \Nc_{\Omega}(\gamma_n^{-1} S;\varepsilon)
\end{align*}
for $n$ sufficiently large, which contradicts Equation \eqref{eqn:contradiction-step-for-cdim2} and concludes the proof of this lemma.
\end{proof}

We will now use Lemma \ref{lem:cdim2-implies-simplex} to show that there exists $S_0 \in \Sc_0$ such that $w \in F_{\Omega}(\partial S_0)$. 

Since $\Sc_0$ is coarsely complete, there exists $R_0\geq 0$ such that any properly embedded simplex of dimension at least two in $\Cc$ is contained in the $R_0$-tubular neighborhood of a simplex in $\Sc_0$. Fix $\varepsilon > 0$. Since $\Sc_0$ is strongly isolated, there exists $D_{\varepsilon} \geq 0$ such that: if $S_1, S_2 \in \Sc_0$ are distinct, then 
\begin{align}
\label{eqn:strongly-isolated-D}
\diam_\Omega \left( \Nc_{\Omega}(S_1;\varepsilon+R_0) \cap \Nc_{\Omega}(S_2; \varepsilon+R_0) \right) \leq D_{\varepsilon}.
\end{align}

Fix $r:=D_{\varepsilon}+1$ and any point $p \in \Cc$. Apply Lemma \ref{lem:cdim2-implies-simplex} to $r, \varepsilon$ and $p$ to get $w_0 \in (w,w^\prime)$ and $p_0 \in [p,w)$ satisfying the conclusions of the lemma. Then pick a sequence $(x_n)_{n \geq 1}$ in  $[p_0,w_0)$ such that $x_n \to w_0$ and 
\begin{align*}
\dist_{\Omega}(x_n,x_{n+1})=r
\end{align*}
for all $n \geq 1$. By Lemma \ref{lem:cdim2-implies-simplex} and our choice of $R_0>0$, for each $n$ there exists a properly embedded simplex $S_n \in \Sc_0$ such that
\begin{align*}
\Pb(\Spanset\{ w, w^\prime, p \}) \cap \Bc_{\Omega}(x_n;r) \subset \Nc_{\Omega}(S_n;\varepsilon+R_0).
\end{align*}
Then, if $n \geq 1$, 
\begin{align*}
(x_n,x_{n+1}) &\subset \Bc_{\Omega}(x_n;r) \cap  \Bc_{\Omega}(x_{n+1};r) \cap \Pb \left( \Spanset\{ w, w^\prime, p \} \right) \\
&\subset \Nc_{\Omega}(S_n;\varepsilon+R_0) \cap \Nc_{\Omega}(S_{n+1};\varepsilon+R_0). 
\end{align*}
Thus 
\begin{align*}
\diam_\Omega \left( \Nc_{\Omega}(S_n;\varepsilon+R_0) \cap \Nc_{\Omega}(S_{n+1};\varepsilon+R_0) \right) \geq \dist_{\Omega}(x_n,x_{n+1})=r> D_{\varepsilon}.
\end{align*}
Then Equation \eqref{eqn:strongly-isolated-D} implies that $S_n=S_{n+1}=:S_0$ for all $n \geq 1$. Then $\{x_n : n \in \Nb\} \subset \Nc_{\Omega}(S_0;\varepsilon+R_0)$ and so by Proposition \ref{prop:dist_est_and_faces}, 
\begin{align*}
w_0=\lim_{n \rightarrow \infty} x_n \in F_{\Omega}(\partial S_0).
\end{align*}
Then $w \in F_{\Omega}(\partial S_0)$ as $w_0 \in F_{\Omega}(w)$. By Theorem \ref{thm:ncc_rel_hyp} part (6), there exists $D^\prime>0$ such that
\begin{align*}
\partiali \Cc \cap F_{\Omega}(w) \subset \Nc_{F_{\Omega}(w)}\left( F_{S_0}(w); D^\prime \right),
\end{align*}
that is, 
\begin{align*}
 \cdim_{F_{\Omega}(w)}\left( \partiali \Cc \cap F_{\Omega}(w) \right) \leq \dim F_{S_0}(w)  \leq \dim S_0-1=1.
\end{align*}
This proves that for any $w \in \partiali \Cc$, 
\begin{align*}
 \cdim_{F_{\Omega}(w)}\left( \partiali \Cc \cap F_{\Omega}(w) \right) \leq 1.
\end{align*}

\subsection{(1) $\Rightarrow$ (2)} Suppose $\cdim_{F_\Omega(x)} \left(\partiali\Cc \cap F_\Omega(x)\right) \leq 1$ for all $x \in \partiali\Cc$.

 Let $\Sc_0$ denote the collection of all properly embedded simplices in $\Cc$ with dimension at least two. By Observation~\ref{obs:qdim_of_faces_PES}, $\Cc$ does not contain any properly embedded simplices with dimension three or more. Hence $\Sc_0$ consists of two dimensional simplices. 

We will construct a collection $\Sc \subset \Sc_0$ of properly embedded two dimensional simplices which are isolated, coarsely complete, and $\Gamma$-invariant (see Definition~\ref{defn:IS}). Then Theorem~\ref{thm:main_char_ncc} will imply that $(\Cc, \dist_\Omega)$ is relatively hyperbolic with respect to a family of properly embedded simplices in $\Cc$. 

We note that it is possible for $\Sc_0$ to have non-discrete families of parallel maximal properly embedded simplices (see Lemma~\ref{lem:slide_along_faces} and \cite[Section 2.3]{IZ2019b}) and hence the challenge in constructing $\Sc$ is to identify a ``canonical'' simplex in each family of parallel simplices. This is accomplished by using a center of mass construction, which is similar to the construction of $\Sc_{\rm core}$ in the proof of Theorem 10.1 in~\cite{IZ2019b}.

Since the proof is lengthy, we provide a short outline of the steps involved. First we prove a technical result, \cref{lem:faces_of_vertices}, which implies that each family of parallel simplices is uniformly bounded (also see Lemma~\ref{lem:slide_along_faces}). This uniformity is key in the center of mass construction in Equations \eqref{eqn:center_of_mass_for_parallel_simplices} and \eqref{eqn:canonical_simplex_construction}. As mentioned above, this construction identifies one ``canonical" simplex in each family of parallel simplices. Once this ``canonical" set of simplices is constructed, the rest of the section (\cref{lem: S is gamma invariant in final argument} to \cref{lem: isolated in final argument}) is devoted to verifying that this family is indeed isolated, coarsely complete, and $\Gamma$-invariant. These lemmas are analogues in the naive convex co-compact case of \cref{lem:lines_intersecting} through \cref{lem:cc_collection_is_isolated}. The former lemmas play a similar role here as the latter lemmas did in the proof of (1) $\implies$ (2) of \cref{thm:cc_one_d_faces}.

We now being our proof. The key idea behind the proof of the next lemma is the following. If the lemma fails, we can use a re-scaling argument to construct a properly embedded two-dimensional simplex $S$ with a vertex $a$ such that $\cdim_{F_{\Omega}(a)}F_{\Omega}(a) \cap \partiali \Cc=1$. We can then construct  a boundary face of coarse dimension two and reach a contradiction.

\begin{lemma}\label{lem:faces_of_vertices} There exists $R > 0$ such that: if $S \in \Sc_0$ and $a \in \partial S$ is a vertex of $S$, then 
\begin{align*}
\diam_{F_\Omega(a)} \Big( \partiali\Cc \cap F_\Omega(a)\Big) \leq R.
\end{align*}

\end{lemma}

\begin{proof} Suppose not. Then for each $n \geq 1$, there exists a properly embedded properly embedded two dimensional simplex $S_n \subset \Cc$ with a vertex $a_n \in \partial S_n$ where
\begin{align*}
\diam_{F_\Omega(a_n)} \Big( \partiali\Cc \cap F_\Omega(a_n)\Big) > n. 
\end{align*}
So there exists $a_n^\prime, a_{n}^{\prime\prime} \in\partiali\Cc \cap  F_\Omega(a_n)$ with 
\begin{align*}
\dist_{F_\Omega(a_n)}( a_n^\prime, a_n^{\prime\prime}) \geq n. 
\end{align*}
Using Lemma~\ref{lem:slide_along_faces} we can assume that $a_n$ is the $\dist_{F_{\Omega}(a_n)}$ Hilbert distance midpoint of $[a_n^\prime, a_n^{\prime\prime}]$. 

Let $b_n, c_n \in \partial S_n$ be the other vertices of $S_n$. Then Lemma~\ref{lem:slide_along_faces} implies that 
\begin{align*}
S_n^\prime :=\relint {\rm ConvHull}_{\overline{\Omega}}(a_n^\prime, b_n, c_n) 
\end{align*}
and 
\begin{align*}
S_n^{\prime\prime} := \relint {\rm ConvHull}_{\overline{\Omega}}(a_n^{\prime\prime}, b_n, c_n)
\end{align*}
are properly embedded simplices in $\Cc$ with
\begin{align*}
\dist_\Omega^{\Haus}\left( S_n, S_n^\prime \right) \leq \dist_{F_\Omega(a_n)}(a_n,a_n^\prime)
\end{align*}
and 
\begin{align*}
\dist_\Omega^{\Haus}\left( S_n, S_n^{\prime\prime} \right) \leq \dist_{F_\Omega(a_n)}(a_n,a_n^{\prime\prime}).
\end{align*}

\noindent \fbox{Claim:} For each $n \geq 1$ there exists $p_n \in S_n$ with 
\begin{align}
\label{eqn:y_n_far_from_simplices}
\min\{\dist_\Omega(p_n, S_n^\prime), \, \dist_\Omega(p_n, S_n^{\prime\prime})\} \geq n/2-1.
\end{align}

Fix $n$ and a point $x_n \in (b_n, c_n)$. Then fix a sequence $(q_{m})_{m \geq 1}$ in $(a_n,x_n)$ converging to $a_n$. For each $m$, fix $q_m^\prime \in S_n^\prime$ with 
$$
\dist_\Omega(q_m, S_n^\prime) = \dist_\Omega(q_m, q_m^\prime). 
$$
Since $\dist_\Omega^{\Haus}(S_n,S_n^\prime) \leq \dist_{F_{\Omega}(a_n)}(a_n,a_n')$, we have
\begin{align*}
\hil(q_{m}, q_{m}^\prime) \leq \dist_{F_\Omega(a_n)}(a_n,a_n^\prime)
\end{align*}
for all $m \geq 1$.

Since $q_m \rightarrow a_n$, the above estimate and Proposition~\ref{prop:dist_est_and_faces} imply that any limit point of $(q_{m}^\prime)_{m \geq 1}$ is in $F_\Omega(a_n) \cap \partial S_n^\prime = \{ a_n^\prime\}$. Thus, up to passing to a subsequence, $\lim_{m \rightarrow \infty} q_{m}^\prime = a_n^\prime$. Then Proposition~\ref{prop:dist_est_and_faces} implies that
\begin{align*}
\frac{n}{2} \leq \dist_{F_\Omega(a_n)}(a_n,a_n^{\prime}) \leq \liminf_{m \rightarrow\infty} \dist_\Omega( q_{m}, q_{m}^\prime) = \liminf_{m \rightarrow\infty} \dist_\Omega( q_{m}, S_n^\prime).
\end{align*}
So for $m$ sufficiently large $\frac{n}{2}-1 \leq \dist_\Omega( q_{m}, S_n^\prime)$. The same reasoning shows that $\frac{n}{2}-1 \leq \dist_\Omega( q_{m}, S_n^{\prime\prime})$ when $m$ is large. So $p_n := q_m$ for $m$ large enough satisfies the claim. This finishes the proof of this claim.

By passing to a subsequence and translating by $\Gamma$ we can assume that $p_n \rightarrow p \in \Cc$. Passing to further subsequences we can suppose that 
\begin{align*}
a_n, a_n^\prime, a_n^{\prime\prime}, b_n, c_n \rightarrow a,a^\prime, a^{\prime\prime}, b, c \in \partiali \Cc. 
\end{align*}
By construction $[a,b] \cup [b,c] \cup [c,a] \subset \partiali\Cc$ while  $p \in \relint {\rm ConvHull}_{\overline{\Omega}}(a,b,c) \cap \Cc$. So 
\begin{align}
\label{eqn:S_in_Omega} 
S:= \relint {\rm ConvHull}_{\overline{\Omega}}(a,b,c) \subset \Omega
\end{align}
is a properly embedded simplex in $\Cc$. Equation~\eqref{eqn:y_n_far_from_simplices} implies that 
\begin{align}
\label{eqn:Sprime_not_in_Omega}
{\rm ConvHull}_{\overline{\Omega}}(a^\prime,b,c) \cup {\rm ConvHull}_{\overline{\Omega}}(a^{\prime\prime},b,c) \subset \partiali\Cc.
\end{align}

By construction, $a_n \in [a_n^\prime, a_n^{\prime\prime}]$ for all $n$ and so $a \in [a^\prime, a^{\prime\prime}]$. Observation~\ref{obs:faces} part (4) and Equations~\eqref{eqn:S_in_Omega},~\eqref{eqn:Sprime_not_in_Omega} imply that $a^\prime \neq a^{\prime\prime} \in \partial F_\Omega(a)$. So $L:=(a^\prime, a^{\prime\prime})$ is a properly embedded one dimensional simplex in $\partiali\Cc \cap F_\Omega(a)$. Thus Observation \ref{obs:qdim_of_faces_PES} implies that  
\begin{align*}
\cdim_{F_\Omega(a)} \left( \partiali\Cc \cap F_\Omega(a)\right) \geq \cdim_{F_\Omega(a)} \left( L\right)  = \dim( L) = 1.
\end{align*}

Now fix a point $x \in \partial S$ in the relative interior of an edge adjacent to $a$, then $a \in \partial F_\Omega(x)$ by Observation~\ref{obs:faces_of_simplices_are_properly_embedded}. So Observation~\ref{obs:lower_bd_from_bd}  applied to $\partiali\Cc \cap F_\Omega(x) \subset F_\Omega(x)$ yields
\begin{align*}
\cdim_{F_\Omega(x)} \left(\partiali\Cc \cap F_\Omega(x)\right) 
&
\geq 1+ \cdim_{F_{F_{\Omega}(x)}(a)}\left( \partiali\Cc \cap \partial F_{\Omega}(x) \cap F_{F_{\Omega}(x)}(a)\right)
\\
& = 1+ \cdim_{F_\Omega(a)} \left(\partiali\Cc \cap F_\Omega(a)\right) \\
& \geq 2.
\end{align*}
This is a contradiction to our hypothesis that $\cdim_{F_\Omega(x)} \left(\partiali\Cc \cap F_\Omega(x)\right) \leq 1$ for all $x \in \partiali\Cc$.
\end{proof}

Next we define a map $\Phi : \Sc_0 \rightarrow \Sc_0$ which maps parallel simplices to a single simplex. Suppose $S \in \Sc_0$ has vertices $v_1,v_2,v_3$. By the above lemma, $\partiali \Cc \cap F_{\Omega}(v_i)$ is a compact subset of $F_{\Omega}(v_i)$ for $i=1,2,3$. Then using the center of mass from Proposition~\ref{prop:center_of_mass}, define
\begin{align}
\label{eqn:center_of_mass_for_parallel_simplices}
w_j := {\rm CoM}_{F_\Omega(v_j)}( \partiali\Cc \cap F_\Omega(v_j) )
\end{align}
and 
\begin{align}
\label{eqn:canonical_simplex_construction}
\Phi(S) := \relint {\rm ConvHull}_{\overline{\Omega}}( w_1,w_2,w_3).
\end{align}
 Then $\Phi(S)$ is a properly embedded two dimensional simplex in $\Cc$ by Lemma~\ref{lem:slide_along_faces}. Then define 
\begin{align*}
\Sc: = \{ \Phi(S) : S \in \Sc_0\}.
\end{align*}

 The next two lemmas verify that $\Sc$ is $\Gamma$-invariant and coarsely complete.

\begin{lemma}\label{lem: S is gamma invariant in final argument} $\Sc$ is $\Gamma$-invariant. \end{lemma} 

\begin{proof} Since $\Sc_0$ is $\Gamma$-invariant, this follows from the equivariance of the center of mass. \end{proof}

  \begin{lemma}\label{lem:bd_distance_parallel_simplices} If $S_1, S_2 \subset \Cc$ are properly embedded two dimensional simplices and $\Phi(S_1) = \Phi(S_2)$, then 
  \begin{align*}
  \dist_\Omega^{\Haus}(S_1, S_2) \leq R.
  \end{align*}
  In particular, $\Sc$ is coarsely complete. 
  \end{lemma}
  
  \begin{proof} Note that $\Phi(S_1)=\Phi(S_2)$ implies that $S_1$ and $S_2$ are parallel simplices. The first assertion then follows immediately from Lemma~\ref{lem:faces_of_vertices} and Lemma~\ref{lem:slide_along_faces}. For the in particular part, suppose $S \subset \Cc$ is a properly embedded two dimensional simplex. Then $\Phi(S) = \Phi(\Phi(S))$ and so by the first part
    \begin{align*}
  \dist_\Omega^{\Haus}(S, \Phi(S)) \leq R.
  \end{align*}
  Thus $S \subset \Nc_\Omega(\Phi(S); R)$. 
  \end{proof}

The proof that $\Sc$ is isolated is more involved and requires two preliminary lemmas.

\begin{lemma}\label{lem:lines_intersecting_faces} If $\ell \subset \partiali \Cc$ is a line segment, $S$ is a properly embedded two dimensional simplex, and $\ell \cap F_\Omega(\partial S) \neq \emptyset$, then $\ell \subset F_\Omega(\partial S)$. \end{lemma}

\begin{proof} It is enough to consider the case where $\ell=[x,y] \subset \partiali \Cc$ and $S$ is a properly embedded two dimensional simplex
$S \subset \Cc$ with $x \in F_\Omega(\partial S)$. Using Observation~\ref{obs:faces} we may assume that $x \in \partial S$. Indeed, by definition there exists $x_0 \in \partial S$ such that $x \in F_{\Omega}(x_0)$. Then the projective line segment $\ell_0:=[x_0,y] \subset \partiali \Cc$ also satisfies our assumptions and Observation~\ref{obs:faces} implies that $\ell \subset F_\Omega(\ell_0)$. Hence, without loss of generality, we will make the simplifying assumption that $x \in \partial S$. 

Now suppose, for a contradiction, that $\ell$ is not contained in $F_\Omega(\partial S)$. Since $\ell$ is not contained in $F_\Omega(\partial S)$ we must have $y \notin F_\Omega(\partial S)$.

Recall that by hypothesis, $\cdim_{F_{\Omega}(x')}(\partiali \Cc \cap F_{\Omega}(x')) \leq 1$ for any $x' \in \partiali \Cc$. Our proof will be a case-by-case analysis where we arrive at a contradiction in each case by finding a point in $\partiali \Cc$ where the above hypothesis on coarse dimension fails. Since $x \in \partial S$, there are two cases to consider based on whether $x$ is a vertex of $S$ or $x$ is contained in an edge of $S$.

\medskip 

\noindent \fbox{Case 1:} Assume $x$ is contained in an edge of $S$. 

\medskip

In this case, fix some $m \in (x,y)$. Then $m \in \partiali \Cc$ and there are two sub-cases to consider depending on whether $x\in \partial F_{\Omega}(m)$ or $x \in F_{\Omega}(m)$.

\medskip
\noindent \fbox{Case 1 (a):} Assume $x \in \partial F_\Omega(m)$.  In this case, we will arrive at a contradiction by showing that the coarse dimension of $\partiali\Cc \cap F_{\Omega}(m)$ is at least $2$.

To this end, we first apply \cref{obs:lower_bd_from_bd} to the properly convex domain $F_{\Omega}(m)$ in $\Pb(\Rb^{d'})$, where $d':=\dim F_{\Omega}(m)$, and the non-empty convex subset $\partiali\Cc \cap ~F_{\Omega}(m) \subset F_{\Omega}(m)$. Note that in this case, $\partiali ( \partiali\Cc \cap F_{\Omega}(m))=\partiali \Cc \cap ~\partial F_{\Omega}(m)$. Thus \cref{obs:lower_bd_from_bd} yields,
\begin{align}
\label{eq:c-dim-ineq-bdry-face-of-m}
\cdim_{F_\Omega(m)}&(\partiali \Cc \cap F_\Omega(m)) \\ &\geq 1+ \cdim_{F_{F_{\Omega}(m)}(x)} \left( \partiali\Cc \cap ~\partial F_{\Omega}(m) \cap F_{F_{\Omega}(m)}(x) \right).\nonumber
\end{align}

We now claim that $$F_S(x) \subset \partiali \Cc \cap ~\partial F_{\Omega}(m) \cap F_{F_{\Omega}(m)}(x).$$
To prove the claim, first observe that $F_{F_{\Omega}(m)}(x)=F_{\Omega}(x)$. Then the only non-trivial part in the claim is to show that $F_S(x) \subset \partial F_{\Omega}(m)$. Indeed, since $x \in \partial F_{\Omega}(m)$, \cref{obs:faces} part (3) implies that $F_{\Omega}(x) \subset \partial F_{\Omega}(m)$. Since $\Sc \subset \Omega$ is properly embedded, $F_S(x) \subset F_{\Omega}(x)$ and thus $F_S(x) \subset \partial F_{\Omega}(m)$. 

Then the above claim and the inequality in \eqref{eq:c-dim-ineq-bdry-face-of-m} imply that 
\begin{align*}
\cdim_{F_\Omega(m)}(\partiali \Cc \cap F_\Omega(m)) &\geq 1+ \cdim_{F_{F_{\Omega}(m)}(x)} \left( F_S(x)\right)\\
&=1+ \cdim_{F_{\Omega}(x)} \left( F_S(x)\right).
\end{align*}
By Observation \ref{obs:qdim_of_faces_PES}, $\cdim_{F_{\Omega}(x)}(F_S(x))=\dim(F_S(x))=1.$ Thus
\begin{align*}
\cdim_{F_\Omega(m)}(\partiali \Cc \cap F_\Omega(m)) \geq 2
\end{align*}
and we have a contradiction.


\medskip 

\noindent \fbox{Case 1 (b):} Assume $x \in F_\Omega(m)$ or equivalently $m \in F_\Omega(x)$.  In this case again, we will arrive at a contradiction by showing that the coarse dimension of $\partiali\Cc \cap F_{\Omega}(m)$ is at least $2$.

Recall that $x \in S$ and $y \notin F_\Omega(\partial S)$. Since $(x,y) \subset \partial \Omega$, then we must have $y \in \partial F_\Omega(x)$. Let $v_1, v_2 \in \partial S$ be the vertices of the edge containing $x$. Then by Observation~\ref{obs:faces_of_simplices_are_properly_embedded} 
\begin{align*}
v_1, v_2 \in \partiali \Cc \cap \partial F_\Omega(m)
\end{align*}
and $F_{\Omega}(v_1)$ and $F_{\Omega}(v_2)$ are distinct.
Further, since $y \notin F_\Omega(\partial S)$ the faces $F_\Omega(v_1)$, $F_\Omega(v_2)$, $F_\Omega(y)$ are all distinct. 

Finally, we will apply \cref{obs:lower_bd_from_three_different_faces} to the properly convex domain $F_{\Omega}(m)$ in $\Pb(\Rb^{d'})$, where $d':=\dim F_{\Omega}(m)$, and the non-empty convex subset $\partiali\Cc \cap ~F_{\Omega}(m) \subset F_{\Omega}(m)$. The three points in $F_{\Omega}(m)$ that we consider are $v_1, v_2,$ and $y$.  Since $F_{F_{\Omega}(m)}(\cdot)=F_{\Omega}(\cdot)$ for any point in $\overline{F_{\Omega}(m)}$, the faces in $F_{\Omega}(m)$ of the three points $v_1,v_2,$ and $y$ are pairwise distinct. Note that in this case, $\partiali ( \partiali\Cc \cap F_{\Omega}(m))=\partiali \Cc \cap ~\partial F_{\Omega}(m)$. Thus, by Observation~\ref{obs:lower_bd_from_three_different_faces},  we have 
\begin{align*}
\cdim_{F_\Omega(m)} (\partiali \Cc \cap F_\Omega(m)) \geq 2
\end{align*}
and hence a contradiction.


\medskip 

\noindent \fbox{Case 2:} Assume $x$ is a vertex of $S$.  
In this case, we will arrive at a contradiction by finding a point $w \in \partiali \Cc$ for which the coarse dimension of  $\partiali \Cc \cap F_{\Omega}(w)$ is at least 2. In particular, we will use Proposition~\ref{prop:finding_two_d_faces} to find such a point $w$.

Let $y_1, y_2 \in \partial S$ be points on the edges adjacent to $x$. Then $(y_1,y_2) \subset \Omega$. Then 
\begin{align*}
[x,y_1] \cup [x,y_2] \cup [x,y] \subset \partiali\Cc
\end{align*}
and $(y_1, y_2) \subset \Omega$. We claim that $(y_1, y) \subset \Omega$. If not, then we could apply Case 1 to the line segment $\ell^\prime := [y_1, y]$ and obtain a contradiction. So we must have $(y_1, y) \subset \Omega$. By symmetry we also have $(y_2,y) \subset \Omega$. But then by Proposition~\ref{prop:finding_two_d_faces}  there exists $w \in \partiali\Cc$ with 
\begin{align*}
\cdim_{F_\Omega(w)} (\partiali \Cc \cap F_\Omega(w)) \geq 2.
\end{align*}
So we have a contradiction. 
\end{proof}

\begin{lemma}\label{lem:bd_faces_of_triangles_intersecting} If $S_1, S_2 \subset \Cc$ are properly embedded two dimensional simplices and $F_\Omega(\partial S_1) \cap F_\Omega(\partial S_2) \neq \emptyset$, then $\Phi(S_1) = \Phi(S_2)$. \end{lemma}

\begin{proof} Lemma~\ref{lem:lines_intersecting_faces} implies that $F_\Omega(\partial S_1) = F_\Omega(\partial S_2)$. Suppose $v_1,v_2,v_3 \in \partial S_1$ are the vertices of $S_1$. Then there exist $w_1,w_2,w_3 \in \partial S_2$ such that $F_\Omega(v_j) = F_\Omega(w_j)$. Then Lemma~\ref{lem:faces_of_vertices} and Observation~\ref{obs:faces_of_simplices_are_properly_embedded} implies that $w_1,w_2,w_3$ are the vertices of $S_2$. So by definition $\Phi(S_1)=\Phi(S_2)$. 
  \end{proof}

\begin{lemma}\label{lem: isolated in final argument}  $\Sc$ is isolated, that is $\Sc$ is closed and  discrete in the local Hausdorff convergence topology. \end{lemma}

\begin{proof}

By Proposition~\ref{prop:PES_closed} the collection $\Sc_0$ of all properly embedded two dimensional simplices in $\Cc$ is closed in the local Hausdorff convergence topology. So to show that $\Sc$ is closed and  discrete in the local Hausdorff convergence topology, it is enough to fix a sequence $(S_n)_{n \geq 1}$ in $\Sc$ such that $S_n$ converges in the local Hausdorff convergence topology to a properly embedded two dimensional simplex $S$ and then show that $S_n=S$ for $n$ sufficiently large. 

Suppose not, then by passing to a subsequence we can suppose that $S_n \neq S$ for all $n$. Fix $p_0 \in S$. Then for $n \geq 0$ let 
\begin{align*}
R_n: = \sup\left\{ r \geq 0 : S \cap \Bc_\Omega(p_0; r) \subset \overline{\Nc_\Omega(S_n; R+1)} \right\}
\end{align*}
where $R>0$ is as in the statements of Lemma \ref{lem:faces_of_vertices} and Lemma  \ref{lem:bd_distance_parallel_simplices}. After passing to a subsequence, we can consider the following two cases. 

\medskip 

\noindent \fbox{Case 1:} Assume $R_n = \infty$ for all $n$. Then for any $n$, 
\begin{align*}
S\subset \overline{\Nc_\Omega(S_n; R+1)}
\end{align*}
and so by Proposition~\ref{prop:dist_est_and_faces} 
\begin{align*}
\partial S \subset F_\Omega(\partial S_n).
\end{align*}
Then Lemma~\ref{lem:bd_faces_of_triangles_intersecting} implies that $\Phi(S) = \Phi(S_n)=S_n$ for all $n$. Since $S_n \rightarrow S$, we then have $S=\Phi(S) = S_n$ for all $n$. So we have a contradiction. 

\medskip 

\noindent \fbox{Case 2:} Assume $R_n < \infty$ for all $n$. Since $S_n \rightarrow S$ in the local Hausdorff convergence topology, we see that $R_n \rightarrow \infty$  (see \cref{obs:convergence_in_loc_haus_top}). Then there exists a sequence $(q_n)_{n \geq 1}$ in $S$ such that 
\begin{enumerate}
\item $\lim_{n \rightarrow \infty} \dist_\Omega(q_n,p_0) = \infty$, 
\item $[q_n, p_0] \subset  \overline{\Nc_\Omega(S_n; R+1)}$, and 
\item $\dist_\Omega(q_n, S_n) =R+1$. 
\end{enumerate}

Next pick $\gamma_n \in \Gamma$ such that $\{ \gamma_n q_n : n \geq 0\}$ is relatively compact in $\Cc$. Then by passing to a subsequence we can suppose that $\gamma_n q_n \rightarrow q \in \Cc$ and $\gamma_n p_0 \rightarrow p \in \partiali \Cc$. Using Proposition~\ref{prop:PES_closed} and passing to another subsequence we can suppose that $\gamma_n S_n \rightarrow S^\prime$ and $\gamma_n S \rightarrow S^{\prime\prime}$ where $S^\prime$ and $S^{\prime\prime}$ are properly embedded two dimensional simplices in $\Cc$. Further, 
\begin{align*}
[q,p) \subset S^{\prime\prime} \cap \overline{\Nc_\Omega(S^\prime; R+1)}.
\end{align*}
Then Proposition~\ref{prop:dist_est_and_faces} implies that $p \in \partial S^{\prime\prime} \cap F_\Omega(S^\prime)$. So $\Phi(S^\prime) = \Phi(S^{\prime\prime})$ by Lemma~\ref{lem:bd_faces_of_triangles_intersecting}. However, by construction $q \in S^{\prime\prime}$ and $\dist_\Omega(q,S^\prime) = R+1$. So we have a contradiction with Lemma~\ref{lem:bd_distance_parallel_simplices}.
\end{proof}

Thus $\Sc$ is isolated, coarsely complete, and $\Gamma$-invariant by Lemmas~\ref{lem: S is gamma invariant in final argument},~\ref{lem:bd_distance_parallel_simplices}, and~\ref{lem: isolated in final argument}. Then Theorem~\ref{thm:main_char_ncc}  implies that $(\Cc, \dist_\Omega)$ is relatively hyperbolic with respect to a family $\Sc_{\diamond}$ of properly embedded simplices in $\Cc$ of dimension at least two. Note that $\Sc_{\diamond}$ isn't necessarily $\Sc$; see the discussion following \cref{thm:main_char_ncc}.

Since $\cdim_{F_\Omega(x)} \left(\partiali\Cc \cap F_\Omega(x)\right) \leq 1$ for all $x \in \partiali\Cc$, Observation~\ref{obs:qdim_of_faces_PES} implies that $\Cc$ does not contain any properly embedded simplices with dimension three or more. So each simplex in $\Sc_{\diamond}$ is two-dimensional.  Thus $(\Cc, \dist_\Omega)$ is relatively hyperbolic with respect to a collection of two dimensional properly embedded simplices. This completes the proof of this direction.

\appendix

\section{Proof of Observation~\ref{obs:faces}} 
\label{sec:proof_of_obs_about_faces}

At the request of one of the referees, we include a proof of Observation~\ref{obs:faces} which we restate here. 

\begin{observation}
Suppose $\Omega \subset \Pb(\Rb^d)$ is a properly convex domain. 
\begin{enumerate}
\item $F_\Omega(x)$ is convex and open in its span,
\item $y \in F_\Omega(x)$ if and only if $x \in F_\Omega(y)$ if and only if $F_\Omega(x) = F_\Omega(y)$,
\item if $y \in \partial F_\Omega(x)$, then $F_\Omega(y) \subset \partial F_\Omega(x)$,
\item if $x, y \in \overline{\Omega}$, $z \in (x,y)$, $p \in F_{\Omega}(x)$, and $q \in F_{\Omega}(y)$, then 
\begin{align*}
(p,q) \subset F_\Omega(z).
\end{align*}
In particular, $(p,q) \subset \Omega$ if and only if $(x,y) \subset \Omega$.
\end{enumerate}
\end{observation}

For the rest of the section, fix a properly convex domain $\Omega \subset \Pb(\Rb^d)$. 

\begin{lemma} If $x \in \overline{\Omega}$ and $y \in F_\Omega(x)$, then $F_\Omega(x) = F_\Omega(y)$. \end{lemma} 

\begin{proof} We start by showing that $F_\Omega(x) \subset F_\Omega(y)$. To that end, fix $z \in F_\Omega(x)$ and let $V : = \Spanset\{x,y,z\}$. If $\dim V \leq 2$, then it is clear that $z \in F_\Omega(y)$. So suppose that $\dim V = 3$. Then we can fix coordinates on $V$ so that 
$$
x=[1:0:0], \quad y = [1:1:0], \quad z = [1:0:1]. 
$$
Since $y,z \in F_\Omega(x)$, there exists $\epsilon > 0$ such that 
$$
 [1:-\epsilon:0],  [1:1+\epsilon:0],  [1:0:-\epsilon],  [1:0:1+\epsilon] \in \overline{\Omega}.
 $$
 Since the convex hull of these points are in $\overline{\Omega}$, we see that $z \in F_\Omega(y)$. Hence $F_\Omega(x) \subset F_\Omega(y)$.
 
Then $x \in F_\Omega(x) \subset F_\Omega(y)$ and so the above argument implies that $F_\Omega(y) \subset F_\Omega(x)$. 

\end{proof} 

\begin{proof}[Proof of (1):] We first show that $F_\Omega(x)$ is convex. Fix $y,z \in F_\Omega(x)$. Then by the lemma, $z \in F_\Omega(y)$ and so $[y,z] \subset F_\Omega(y) = F_\Omega(x)$. So $F_\Omega(x)$ is convex. Then by definition $F_\Omega(x)$ is open in its span. 
\end{proof} 

\begin{proof}[Proof of (2):] This follows immediately from the lemma. 
\end{proof} 

\begin{proof}[Proof of (3):] Since $F_\Omega(y) \cap F_\Omega(x) = \emptyset$, it suffices to show that $F_\Omega(y) \subset \overline{F_\Omega(x)}$. To that end, fix $z \in F_\Omega(y)$ and let $V : = \Spanset\{x,y,z\}$. If $\dim V \leq 2$, then $z=y \in \overline{F_\Omega(x)}$. So suppose that $\dim V = 3$. Then we can fix coordinates on $V$ so that 
$$
x=[1:0:0], \quad y = [1:1:0], \quad z = [1:1:1]. 
$$
Since $F_\Omega(x)$ is open in its span and $z \in F_\Omega(y)$, there exists $\epsilon > 0$ such that 
$$
 [1:-\epsilon:0], [1:1:-\epsilon],  [1:1:1+\epsilon] \in \overline{\Omega}.
 $$
  Since the convex hull of these points are in $\overline{\Omega}$, we see that $z \in \overline{F_\Omega(x)}$. Hence $F_\Omega(y) \subset \overline{F_\Omega(x)}$. 
\end{proof} 

\begin{proof}[Proof of (4):] By symmetry it suffices to consider the following cases. 

\medskip 

\noindent \fbox{Case 1:} Assume $F_\Omega(x) = F_\Omega(y)$. In this case, $(p,q) \subset F_\Omega(x)$ and $z \in F_\Omega(x)$. So by part (2),  $(p,q) \subset F_\Omega(x) = F_\Omega(z).$

Then, for the rest of the cases, we may assume that $F_{\Omega}(x) \cap F_{\Omega}(y)=\emptyset$.

\noindent \fbox{Case 2:} Assume $x=p$ and $y = q$. Then $z \in (x,y) = (p,q)$ and so $(p,q) \subset F_\Omega(z)$. 

\medskip 

\noindent \fbox{Case 3:} Assume $x=p$ and $y \neq q$. In this case, fix an open line segment $\ell \subset \overline{\Omega}$ with $y,q \in \ell$. Then the convex hull of $\{x\} \cup \overline{\ell}$ in $\overline{\Omega}$ is a two-dimensional simplex whose relative interior contains $(p,q)$ and $z$. Hence $(p,q) \subset F_\Omega(z)$. 

\medskip

\noindent \fbox{Case 4:} Assume $x \neq p$ and $y \neq q$. In this case, fix open line segments $\ell_1,\ell_2 \subset \overline{\Omega}$ with $x,p \in \ell_1$ and $y,q \in \ell_2$. Then the convex hull of $\overline{\ell}_1 \cup \overline{\ell}_2$ in $\overline{\Omega}$ is either a two-dimensional 4-gon or a three-dimensional simplex. In either case, the relative interior of this convex hull contains $(p,q)$ and $z$. Hence $(p,q) \subset F_\Omega(z)$. 
\end{proof} 

\begin{figure}[h]
\centering
\includegraphics[scale=0.4]{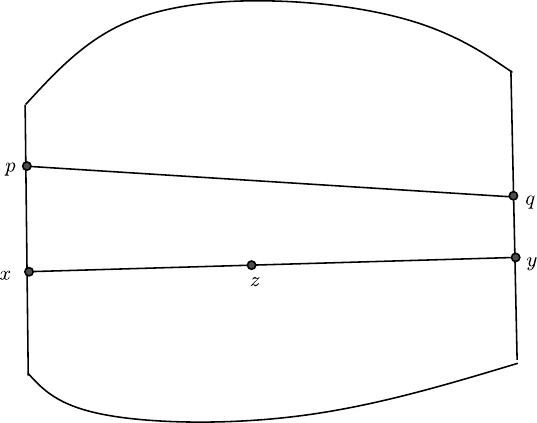}
\caption{Figure for the proof of part (4) Case 4, when  the convex hull of $\overline{\ell}_1 \cup \overline{\ell}_2$ is a two-dimensional 4-gon.}
\label{fig:faces}
\end{figure}

\bibliographystyle{alpha}
\bibliography{geom}

\end{document}